\documentclass[11pt,reqno,a4paper]{amsart}

\usepackage[utf8]{inputenc}

\usepackage{amsthm,amsmath,amsfonts,amssymb}
\numberwithin{equation}{section}  
\usepackage{mathrsfs}
\usepackage{todonotes}

\usepackage{tikz-cd}

\usepackage{mathtools}	 
\usepackage{bm}	  
\usepackage{bbm}	 

\usepackage{tensor}	 
\usepackage[colorlinks=true]{hyperref} 
\usepackage[all]{hypcap}       
 
\newcommand{\C}{\mathbb{C}}
\newcommand{\R}{\mathbb{R}}
\newcommand{\Q}{\mathbb{Q}}
\newcommand{\Z}{\mathbb{Z}}
\newcommand{\ii}{\operatorname{i}}

\newcommand{\ch}{\operatorname{ch}}

\newcommand{\Rea}{\operatorname{Re}}
\newcommand{\Imm}{\operatorname{Im}}

\newcommand{\PP}{\mathbb{P}}
\newcommand{\LL}{\mathbb{L}}

\newcommand{\cA}{\mathcal{A}}
\newcommand{\cB}{\mathcal{B}}

\newcommand{\cY}{\mathcal{Y}}
\newcommand{\cL}{\mathcal{L}}
\newcommand{\cM}{\mathcal{M}}

\newcommand{\cS}{\mathcal{S}}

\newcommand{\cE}{\mathcal{E}}
\newcommand{\cF}{\mathcal{F}}

\newcommand{\cT}{\mathcal{T}}
\newcommand{\cP}{\mathcal{P}}
\newcommand{\cU}{\mathcal{U}}
\newcommand{\olo}{\mathcal{O}}

\newcommand{\shHom}{\mathcal{H}om}
\newcommand{\End}{\operatorname{End}}
\newcommand{\shEnd}{\mathcal{E}nd}
\newcommand{\Bl}{\operatorname{Bl}}

\newcommand{\FS}{D\!\operatorname{FS}}
\newcommand{\Fuk}{D\!\operatorname{Fuk}}

\newcommand{\codim}{\operatorname{codim}}
\newcommand{\Coh}{\operatorname{Coh}}
\newcommand{\relsigma}{\underline{\sigma}}

\newtheorem{thm}{Theorem}[section]
 
\newtheorem{prop}[thm]{Proposition}
\newtheorem{lemma}[thm]{Lemma}
\newtheorem{cor}[thm]{Corollary}
\newtheorem{conj}[thm]{Conjecture}

\theoremstyle{definition}

\theoremstyle{remark}
\newtheorem{exm}[thm]{Example}
\newtheorem{rmk}[thm]{Remark}

\title[Special Lagrangian sections and stability conditions]{Special Lagrangian sections and stability conditions on threefolds}
\author{Jacopo Stoppa}
 
\date{\today}

\begin{document}

\maketitle

\begin{abstract} We study a class of Lagrangian submanifolds, given by sections of a special Lagrangian fibration, contained in certain almost Calabi-Yau threefolds (mirrors of polarised toric threefolds satisfying suitable assumptions). 

We show that, for a Lagrangian section $\cL$ in this class, the shift $\cL[2]$ defines an object in the heart of a natural Bridgeland stability condition on the relevant Fukaya-Seidel category, and that if $\cL[2]$ is \emph{semistable} with respect to this stability condition, then it is isomorphic to a \emph{special} Lagrangian. For mirrors of weak Fanos, the central charge of the stability condition is very close to periods of the holomorphic volume form. These results are consistent with Joyce's interpretation of the Thomas-Yau conjecture. 

As part of the proof we describe a set of line bundles and polarisations on suitable toric threefolds for which semistability with respect to the Bridgeland stability conditions constructed by Bernardara-Macr\`i-Schmidt-Zhao implies the existence of a deformed Hermitian Yang-Mills connection. 
\end{abstract}
\section{Introduction}
The well-known conjectures of Thomas and Yau \cite{Thomas_MomentMirror, ThomasYau} relate the existence of special Lagrangian submanifolds of Calabi-Yau manifolds to notions of stability inspired by algebraic geometry (i.e. Mumford-Takemoto slope stability for vector bundles). 

More recently Joyce \cite{Joyce_ThomasYau} proposed a version of these conjectures in terms of Bridgeland stability. Another key reference on the topic of stability and special Lagrangians is the work of Li \cite{YangLi_ThomasYau}. The latter emphasises the importance of allowing the more general class of \emph{almost} Calabi-Yau manifolds, i.e. complex manifolds  endowed with a K\"ahler form, not necessarily Ricci flat, and a holomorphic volume form.

We recall a partial statement of Joyce's conjecture, since it provides our main motivation, without discussing the details, which will not be relevant to the simplified situation we consider in our work.  
\begin{conj}[Joyce \cite{Joyce_ThomasYau}, Conjecture 3.2]\label{JoyceConj} Let $(M^n, \omega, \Omega)$ be a possibly non-compact Calabi-Yau manifold, with K\"ahler form $\omega$ and holomorphic volume form $\Omega$, suitably convex at infinity. Then there exists a natural stability condition $(Z, \cP)$ on the Fukaya category $\Fuk(M)$, whose objects are Lagrangian branes $\LL$, such that:
\begin{enumerate}
\item[$(i)$] the central charge $Z$ is the composition of natural maps
\begin{equation*}
K_0(\Fuk(M)) \xrightarrow{\LL \mapsto [\cL]} H_{n}(M, \Z) \xrightarrow{[\cL] \mapsto \int_{[\cL]}\Omega} \C,
\end{equation*}
where $\cL \subset M$ is the Lagrangian submanifold underlying the brane $\LL$;
\item[$(ii)$] each isomorphism class of semistable objects of $\Fuk(M)$ of phase $\pi\phi$ (i.e. each isomorphism class in $\cP(\phi)$) contains a unique representative Lagrangian brane $\LL$ such that the underlying Lagrangian submanifold is a (possibly immersed, singular) \emph{special} Lagrangian of phase $e^{\ii \pi \phi}$, that is, a (weak) solution of the equation  
\begin{equation*}
\iota^*_{\cL} \Imm(e^{-\ii \pi \phi} \Omega) = 0. 
\end{equation*} 
\end{enumerate}
\end{conj}
\begin{rmk} As explained by Joyce, formulating $(ii)$ requires a conjectural extension of $\Fuk(M)$ which includes suitable immersed, singular Lagrangians. See \cite{YangLi_ThomasYau}, Section 5 for much progress in this direction. 
\end{rmk}
\begin{rmk} The converse of $(ii)$ is also part of \cite{Joyce_ThomasYau}, Conjecture 3.2, although we do not discuss this here (see \cite{J_toricThomasYau}, Theorem 1.9 for the case of toric surfaces).
\end{rmk}
Given the outstanding difficulty of Conjecture \ref{JoyceConj}, it is natural to look for some radically simplified model case, which nevertheless displays some of the features of the expectations $(i)$, $(ii)$.

One such simplified model was proposed by Collins and Yau in \cite{CollinsYau_momentmaps_preprint}, Section 9, using mirror symmetry for toric manifolds and the Leung-Yau-Zaslow transform for holomorphic line bundles (see also \cite{YangLi_ThomasYau}, Section 2.6 for a discussion). Let us briefly recall some aspects of this.

Suppose $X^n$ is a projective toric manifold. Following the results of Givental \cite{Givental_toric} and their generalisation by several authors, culminating in the work of Coates-Corti-Iritani-Tseng \cite{CoatesCortiIritani_hodge}, a mirror to $X$ is constructed as a family of algebraic tori $\cY \to \cM$, with $\cY_q \cong (\C^*)^n$ for $q \in \cM$, endowed with a regular function $W\!: \cY \to \C$, the Landau-Ginzburg potential. A sufficiently large K\"ahler class $[\omega]$ on $X$ defines a corresponding point $q \in \cM$. The mirror family $\cY \to \cM$ admits important partial compactifications, but we will not use these in this work. Before partial compactification, $\cY \to \cM$ is in fact a trivial fibration, although it is not canonically trivialised.  

On the other hand, following Leung-Yau-Zaslow \cite{LeungYauZaslow} (see also \cite{Chan_survey}, Section 3 and \cite{CollinsYau_momentmaps_preprint}, Section 9 for more details on the toric case), fixing a torus-invariant K\"ahler form $\omega$ on $X$, we obtain a corresponding special Lagrangian torus fibration  $X^o \to \Delta^o$ from the complement $X^o$ of the toric boundary to the open momentum polytope $\Delta^o$. Dualising this torus fibration yields a bounded domain $\cU_q \subset \cY_q$, where $q \in \cM$ corresponds to $[\omega]$. The bounded domain $\cU_q$ is almost Calabi-Yau in a natural way, i.e. it comes with natural mirror K\"ahler form $\omega^{\vee}$ and holomorphic volume form $\Omega$, depending on $\omega$. Explicit formulae for $\omega^{\vee}$, $\Omega$ are given in \cite{CollinsYau_momentmaps_preprint}, Section 9. 

Given $(L, h)$ a holomorphic line bundle on $X$, endowed with a torus-invariant Hermitian metric $h$, Leung-Yau-Zaslow construct a mirror Lagrangian submanifold $\cL = \cL(L, h) \subset \cU_{q}$ with respect to $\omega^{\vee}_q$, which is a section of $\cU_q \to \Delta^o$. Sections obtained from different choices of $h$ are isotopic through a special type of Hamiltonian diffeomorphisms (such that the corresponding holomorphic structure on the line bundle over $X^{o}$ extends to $X$). The Lagrangian $\cL(L, h)$ is called the \emph{real Fourier-Mukai transform}, or \emph{Leung-Yau-Zaslow transform} or, more commonly, the \emph{Strominger-Yau-Zaslow (SYZ) transform} of  the Hermitian holomorphic line bundle $(L, h)$. An explicit formula for $\cL(L, h)$ is given in \cite{CollinsYau_momentmaps_preprint}, Section 9.

Collins and Yau proposed to study versions of the Thomas-Yau conjecture, especially in Joyce's formulation using Bridgeland stability (Conjecture \ref{JoyceConj}), for these Lagrangian sections $\cL = \cL(L, h) \subset (\cU_{q}, \omega^{\vee}, \Omega)$. 

Concretely, one would like to find conditions, in terms of Bridgeland stability, implying that we can find $h$ such that $\cL(L, h) \subset \cU_{q}$ is special Lagrangian. It is well-known that the latter problem is obstructed, as we will recall, see Theorems \ref{LYZThm} and \ref{dHYMThm} (although it is always solvable in the \emph{large volume limit} when we rescale $\omega\mapsto k\omega$ for $k \gg 1$). 

In \cite{J_toricThomasYau} we proved several general results concerning this toric case of the Thomas-Yau conjecture. 

We can now discuss our first main result in the present work. It concerns the case when $X$ is a projective toric threefold, so $(\cU_{q}, \omega^{\vee}, \Omega)$ is three-dimensional almost Calabi-Yau. Since $X$ is not Calabi-Yau, the Fukaya category $\Fuk(M)$ appearing in Conjecture \ref{JoyceConj} must be replaced by some version of the Fukaya-Seidel category $\FS(\cY_q)$ (see Remark \ref{MainThmRmk}, (2)). This involves extra data, namely, the Landau-Ginzburg potential $W_q \!: \cY_q \to \C$ or a conical Lagrangian $\Lambda_q$ in the universal cover of $\cY_q$ (the Lagrangian skeleton). 

Our result relies on rescaling the K\"ahler form $\omega\mapsto k\omega$ by a sufficiently large factor $k > 0$ while at the same time rescaling the Lagrangian $\cL$ by the same factor, i.e. tensoring the mirror line bundle $L \mapsto L^k := L^{\otimes k}$. This is a \emph{large volume, large Lagrangian limit}, which also appears, from a different perspective, in the proposal of Collins and Yau \cite{CollinsYau_momentmaps_preprint}, Section 9. We note that it is very different from the pure large volume limit (keeping $L$ fixed): in particular, the  obstructions mentioned above are \emph{invariant} under this common rescaling (see Remark \ref{ScaleInvarianceRmk}).

We denote by $q_k \in \cM$ the point in the complex moduli space of the mirror corresponding to $k \omega$.

\begin{thm}[Theorem \ref{MainThm}]\label{MainThmIntro} Let $(X, \omega)$ be a smooth projective toric threefold with a Hodge torus-invariant K\"ahler form $\omega$ such that, for all effective divisors $D$, we have
\begin{enumerate}
\item[$(a)$] $\omega \cdot D^2 \geq 0$,
\item[$(b)$] if $\omega \cdot D^2 = 0$, then $D$ lies on an extremal ray of the effective cone.
\end{enumerate}
Fix a line bundle  $L$ on $X$ with Mumford-Takemoto slope $\mu_{\omega}(L) < 0$, satisfying the condition
\begin{equation*}
\omega \cdot (c_1(L))^2 - \frac{\omega^3}{3} > 0. 
\end{equation*} 
Define a graded Lagrangian section $\tilde{\cL}$ in the almost Calabi-Yau manifold $\,\cU_{q_k}$ mirror to $(X, k\omega)$ as a shift of the SYZ transform $\cL(L^k, h) \subset \cU_{q_k} \subset \cY_{q_k}$ with respect to $k\omega$, 
\begin{equation*}
\tilde{\cL} = \cL(L^k, h)[2], 
\end{equation*}
for $k > 0$ sufficiently large (depending only on $(X, \omega)$ and $L$). Then 
\begin{enumerate}
\item[$(i)$] there is a Bridgeland stability condition $\sigma^{\vee} = (Z^{\vee}, \cA^{\vee})$ on the Fukaya-Seidel category $\FS(\cY_{q_k})$ such that $\tilde{\cL}$ lies in the heart $\cA^{\vee}$.
\item[$(ii)$] If $\tilde{\cL}$ is \emph{semistable} with respect to the Bridgeland stability condition $\sigma^{\vee}$ on $\FS(\cY_{q_k})$, and $[\omega]$ is \emph{generic} in the locus of $H^{1,1}(X, \R)$ for which $(a)$, $(b)$ hold, then the isomorphism class of $\tilde{\cL}$ in $\FS(\cY_{q_k})$ contains a \emph{special} Lagrangian, given by a shift of a SYZ section.
\item[$(iii)$] If $X$ is weak Fano (i.e. $-K_X$ is big and nef), the central charge of $\tilde{\cL}$ is given by periods of the holomorphic volume form, up to a small correction term 
\begin{equation*}
Z^{\vee}(\tilde{\cL}) = \frac{1}{(2\pi \ii)^{n}} \int_{[\tilde{\cL}]} e^{-W_{q_k}/z} \Omega_0\,(1 + O(k^{-1})),
\end{equation*}
where $W_{q_k} \!: \cY_{q_k} \to \C$ is the Landau-Ginzburg potential corresponding to $k\omega$, and $[\tilde{\cL}] \in H_n(\cY_{q_k}, \{\Rea(W_{q_k}) \gg 0\}; \Z)$ is a natural integration cycle, coinciding with $\tilde{\cL}$ at least if $L^{\vee}$ is ample.    
\item[$(iv)$] When $X$ is not necessarily weak Fano, there exist a Landau-Ginzburg potential $W_{q_k}$, a \emph{complex} cycle $\Gamma_{\cL}$, with 
\begin{equation*}
[\Gamma_{\cL}] \in H_n(\cY_{q_k}, \{\Rea(W_{q_k}) \gg 0\}; \Z) \otimes \C, 
\end{equation*}
and a holomorphic volume form $\Omega^{(k)}$, such that  
\begin{equation*} 
Z^{\vee}(\tilde{\cL}) = \frac{1}{(2\pi \ii)^{n}}\int_{\Gamma_{\cL}} e^{-W_{q_k}/z}\Omega^{(k)} (1 + O(k^{-1})),  
\end{equation*}
where the integral is understood in the sense of its asymptotic expansion as $z \to 0^+$.
\end{enumerate}
\end{thm}
Theorem \ref{MainThmIntro} is proved in a stronger form (providing more information on the stability condition, potential destabilisers and central charges) in Section \ref{MainSec}, as Theorem \ref{MainThm}.

\begin{exm}[\cite{BernardaraMacri_threefolds}, Section 5] There exist Fano examples, e.g. note that the conditions $(a)$, $(b)$ hold for all polarisations on $X = \PP^1 \times \PP^2$, $X = \PP^1 \times \PP^1 \times \PP^1$ and for polarisations the blowup $X = \Bl_l(\PP^3)$ along a line $l$ of the form $H = a h + b f$ with $a, b > 0$ and $a\leq b$ , where $h$, $f$ are the pullbacks of $\olo_{\PP^3}(1)$ and $\olo_{\PP^1}(1)$ respectively (viewing $X$ as a projective bundle over $\PP^1$).  
\end{exm}

\begin{rmk}\label{MainThmRmk} This result calls for several observations.
\begin{enumerate}
\item The class of polarised toric threefolds $(X, \omega)$ satisfying $(a)$, $(b)$ is precisely the one considered by Bernardara-Macr\`i-Schmidt-Zhao in one of their constructions of stability conditions on threefolds contained in \cite{BernardaraMacri_threefolds} (we recall their result as Theorem \ref{ToricBStabThm}). Under these assumptions on $(X, \omega)$, there is a natural set of Bridgeland stability conditions $\sigma^{\omega}_{\alpha, \beta}$ on $D^b(X)$, parametrised by $\alpha > 0$, $\beta \in \R$. 

Note that the conditions $(a)$, $(b)$ are invariant under rescaling $\omega \mapsto k \omega$; this is crucial in our construction.
\item Toric homological mirror symmetry gives an equivalence  
\begin{equation*}
D^b(X) \cong \FS(\cY_{q_k}).
\end{equation*}
This is due to many authors including \cite{Abouzaid_toricHMS, Zaslow_toricHMS}; we follow \cite{KuwagakiCohCon, SibillaCohCon, Shende_toricMirror, ZhouCohCon} for the non-equivariant case of toric homological mirror symmetry in the sense of Fang-Liu-Treumann-Zaslow \cite{Zaslow_toricHMS}, and the work of Fang \cite{Fang_charges} for Landau-Ginzburg potentials and central charges in the toric weak Fano case appearing in our statement $(iii)$. In fact, mirror symmetry in the sense of Fang-Liu-Treumann-Zaslow is compatible with the Leung-Yau-Zaslow transform for line bundles, as explained e.g. by Fang \cite{Fang_charges}.

Through this equivalence, the stability conditions $\sigma^{\omega}_{\alpha, \beta}$ induce stability conditions on $\FS(\cY_{q_k})$, and we write $\sigma^{\vee}$ for the stability condition on $\FS(\cY_{q_k})$ induced by $\sigma^{k\omega}_{1/\sqrt{3}, 0} = (Z, \cA)$. Note that, in general, constructing stability conditions on Fukaya-type categories seems even harder than the algebro-geometric case of $D^b(X)$. The idea of transferring from $D^b(X)$ seems natural and is mentioned e.g. in \cite{Joyce_ThomasYau}, Remark 3.3 (iii), in the compact Calabi-Yau context.

By construction, the central charge satisfies
\begin{align*}
&Z^{\vee}(\tilde{\cL}) = Z^{\vee}(\cL(L^k, h)[2]) = Z^{\vee}(\cL(L^k, h))\\
& := Z(L^k) = \int_X e^{-\ii k\omega} \ch(L^k).  
\end{align*}
The latter expression is called a central charge of Douglas type. Our statement $(iv)$ follows from general properties of toric mirror symmetry proved in \cite{CoatesCortiIritani_hodge}. However, in this case, we do not know if there is a geometric relation between the Lagrangian submanifold $\cL$ and the cycle $\Gamma_{\cL}$. 
\item The numerical condition $\omega \cdot (c_1(L))^2 - \frac{\omega^3}{3} > 0$ is equivalent to the condition
\begin{equation*}
\arg \int_X (\omega + \ii c_1(L^{\vee}))^n = \hat{\theta} \in \left(\frac{\pi}{2}, \pi\right) \mod 2\pi.
\end{equation*}
As we will see, it implies that we can work with a (mirror) version of the special Lagrangian equation in the so-called \emph{supercritical phase}, for which the best analytic results are available. When $\tilde{\cL}$ is special Lagrangian, its phase angle is given by $\hat{\theta} + 2\pi$.

\item Scaling $\omega \mapsto k\omega$, $L \mapsto L^k$ for large $k > 0$ is essential to our approach for relating to Bridgeland semistability and periods. 

Firstly, Collins-Lo-Yau \cite{CollinsLoYau_K3} and Collins-Shi \cite{Collins_stability} provide examples of Bridgeland stable line bundles $L$ (with respect to $\sigma^{\omega}_{\alpha, \beta}$ or some analogue of these stability conditions on elliptic K3 surfaces) such that $\cL(L, h)$ cannot be equivalent to a special Lagrangian section. 

On the other hand, even when $X$ is a weak Fano threefold, the central charge $Z^{\vee}(\cL(L, h))$ can be very far from a period of $\cL(L, h)$: as we will recall, in this case, the difference from a period is explained by Iritani's Gamma theorem \cite{Iritani_gamma} (see Remark \ref{GammaRmk}).    
\item The works of Collins-Lo-Yau \cite{CollinsLoYau_K3} and Collins-Shi \cite{Collins_stability} mainly study the converse implication (i.e. under what conditions $\cL(L, h)$ being equivalent to special SYZ Lagrangian section implies that $L$ is semistable), in the case when $X$ is a surface, although the results are stated purely on $X$, in terms of (mirror) line bundles, and in particular \cite{CollinsLoYau_K3} works with elliptic K3 surfaces, for which the SYZ transform only holds approximately nearby a large complex structure limit.  
\item The genericity assumption on $\omega$ is explicit: it requires that $[\omega]$ does not lie in the union of the finitely many analytic subvarieties of $H^{1,1}(X, \R)$ cut out by 
\begin{equation*} 
\int_V \Rea(\ii \omega + c_1(L^{\vee}))^{\dim V} - \cot(\varphi)\Imm(\ii \omega + c_1(L^{\vee}))^{\dim V} =  0,
\end{equation*}
as $V$ ranges through irreducible toric subvarieties of $X$, where  
\begin{equation*}
\varphi = \frac{n}{2}\pi - \arg \int_X (\omega + \ii c_1(L^{\vee}))^n \mod 2\pi. 
\end{equation*}
We expect that this genericity assumption on $\omega$ can be removed at the cost of allowing suitable singular special Lagrangian sections. This is motivated by the general variational framework proposed by Li \cite{YangLi_ThomasYau} and by the results of Datar-Mete-Song \cite{DatarSong_slopes} concerning weak solutions of certain complex Monge-Amp\`ere equations on surfaces.  
\item The analogue of Theorem \ref{MainThmIntro} for general projective toric surfaces is proved in \cite{J_toricThomasYau}, Theorem 1.10. The latter result is stated in terms of Bridgeland stability rather than semistability. However one can check that the proof only uses semistability and genericity, just as for Theorem \ref{MainThmIntro}. Moreover, assuming a conjecture by Arcara and Miles, semistability and genericity on surfaces actually imply stability (see \cite{J_toricThomasYau}, Section 6). We do not know under what conditions the same holds for threefolds.  
\end{enumerate}
\end{rmk}

The assumptions $(a)$, $(b)$ in Theorem \ref{MainThmIntro} are quite restrictive. We construct more examples of Lagrangian sections satisfying a Thomas-Yau principle by considering suitable blowups.
\begin{thm}[Theorem \ref{BlpThm}]\label{BlpMainThmIntro} Fix $(X, \omega)$ and $L$ as in Theorem \ref{MainThmIntro}. Let $\pi\!:\tilde{X} = \Bl_p X \to X$ denote the blowup at a torus-fixed point, with exceptional divisor $E$, endowed with the K\"ahler class 
\begin{equation*}
\tilde{\omega}_{\delta} = \pi^*\omega - \delta \sin(\rho)[E]    
\end{equation*}
for $\delta, \rho > 0$ sufficiently small, with $\cos(\rho) \in \Q$. Define the graded Lagrangian section $\hat{\cL}$ in the almost Calabi-Yau manifold $\,\tilde{\cU}_{q_k}$ mirror to $(X, k \tilde{\omega}_{\delta})$,
\begin{equation*}
\hat{\cL} = \cL(\tilde{L}^k, h)[2], 
\end{equation*}
as a shift of the SYZ transform of the Hermitian line bundle
\begin{equation*}
(\tilde{L}^k, h) := (\pi^*L^k(k\cos(\rho)\delta E), h) \subset \tilde{\cU}_{q_k} \subset \tilde{\cY}_{q_k}
\end{equation*}
with respect to a torus-invariant representative of $k\tilde{\omega}_{\delta}$, for $k > 0$ sufficiently large and divisible (depending only on $(X, \omega)$ and $L$).

Then $\hat{\cL}$ lies in the heart of a Bridgeland stability condition $\sigma^{\vee}$ on the Fukaya-Seidel category of the mirror $\FS(\tilde{\cY}_{q_k})$. If $\hat{\cL}$ is \emph{semistable} with respect to $\sigma^{\vee}$ for generic $[\omega] \in H^{1,1}(X, \R)$, it is isomorphic in $\FS(\tilde{\cY}_{q_k})$ to a \emph{special} Lagrangian. Analogues of $(iii)$ and $(iv)$ in Theorem \ref{MainThmIntro} also hold.
\end{thm}
Theorem \ref{BlpMainThmIntro} is proved with a more detailed statement in Section 3 (Theorem \ref{BlpThm}).
\begin{exm} The (only) Fano example is obtained when $\tilde{X} = \Bl_p \PP^3$.
\end{exm}
We also describe a relation between stability conditions and special Lagrangians in a relative setting. Differential-geometrically, this is the well-studied notion of \emph{adiabatic limit} (see e.g. \cite{StoppaTwisted}, Section 4 for an exposition in a similar context). On the mirror side, this corresponds to a special limit in the complex moduli space of $(\cY_{q_k}, W_{q_k})$. The theory of relative stability conditions is developed in great generality in \cite{Bayer_families}.  
\begin{thm}[Theorem \ref{AdiabaticThm}]\label{AdiabaticThmIntro} Let $f\!: X^n \to \PP^1$ be a toric submersion with relative dimension $d = 2, 3$ (so $n = d + 1$). When $d = 3$ suppose the fibres satisfy the assumptions of Theorems \ref{MainThmIntro} or \ref{BlpMainThmIntro}. Let $\omega_m$ be a K\"ahler class on $X$ of the form
\begin{align*}
\omega_m = m f^{*}\omega_{\PP^1} + \omega_X, 
\end{align*}
where $\omega_X$ is relatively K\"ahler and $m > 0$ is sufficiently large. Fix a line bundle $L$ on $X$ such that for the dual $L^{\vee}$ we have  
\begin{equation*}
\arg \int_X (\omega + \ii c_1(L^{\vee}))^n \in \left(\frac{(n-2)\pi}{2}, \frac{(n-1)\pi}{2}\right) \mod 2\pi.
\end{equation*}
Define the graded Lagrangian section $\tilde{\cL}$ in the almost Calabi-Yau manifold $\,\cU_{q_k}$ mirror to $(X, k \omega_m)$,
\begin{equation*}
\tilde{\cL} = \cL(L^k, h)[d-1], 
\end{equation*}
as a shift of the SYZ transform $\cL(L^k, h) \subset \cU_{q_k} \subset \cY_{q_k}$ with respect to $k\omega_m$, for $m, k > 0$ sufficiently large (depending only on $(X, \omega)$ and $L$). 

There is a stability condition $\relsigma$ on $D^b(X)$ over $\PP^1$ (in the sense of \cite{Bayer_families}, Definition 1.1) such that if $L^k[d-1]$ is \emph{$\relsigma$-semistable} and $[\omega_X] \in H^{1,1}(X, \R)$ is generic, then $\tilde{\cL}$ is isomorphic in $\FS(\cY_{q_k})$ to a \emph{special} Lagrangian section. Analogues of $(iii)$ and $(iv)$ in Theorem \ref{MainThmIntro} also hold.
\end{thm}
The parameter $m \gg 1$ corresponds to the adiabatic limit of the metric on $X$.
\begin{rmk}
Theorems \ref{MainThmIntro}, \ref{BlpMainThmIntro} and \ref{AdiabaticThmIntro} are made possible by a key analytic existence result, due to several authors, known as the Nakai-Moishezon criterion for deformed Hermitian Yang-Mills (dHYM) connections; this is discussed in Section \ref{dHYMSec} (Theorem \ref{dHYMThm}). 

In particular, our results rely on showing that, under suitable conditions on threefolds, Bridgeland semistability implies the existence of a dHYM connection, see in particular Corollaries \ref{BStabdHYMCor}, \ref{BlpBstabdHYMCor} and the proof of Theorem \ref{AdiabaticThm} $(i)$. This viewpoint, not involving special Lagrangians, may be preferred by some readers.   
\end{rmk}
\begin{rmk}
Our results, while limited to a particular class of Lagrangian sections, seem to give the first examples on threefolds for which Bridgeland semistability implies the special Lagrangian condition, as in Conjecture \ref{JoyceConj} $(ii)$. As we discussed, our approach is motivated by the works of Collins-Yau \cite{CollinsYau_dHYM_momentmap}, Collins-Lo-Yau \cite{CollinsLoYau_K3} and Collins-Shi \cite{Collins_stability}.  

Li \cite{YangLi_ThomasYau} proves some fundamental results on special Lagrangians and stability, in terms of phase inequalities satisfied by exact triangles in the Fukaya category, under suitable positivity and transversality assumptions. Li also discusses in detail the possible relations to Bridgeland semistability, see in particular \cite{YangLi_ThomasYau}, Section 3.6.

Lotay and Oliveira \cite{Lotay_GibbonsHawkingSLag} prove that the original version of the Thomas conjecture in \cite{Thomas_MomentMirror} holds for circle-invariant compact Lagrangians in the hyperk\"ahler 4-manifolds obtained by the Gibbons-Hawking ansatz. A possible relation to Bridgeland semistability (i.e. Conjecture \ref{JoyceConj}) is mentioned in \cite{Lotay_GibbonsHawkingSLag}, Section 5.4.

Haiden, Katzarkov and Simpson \cite{Haiden_spectral} study similar questions for analogues of special Lagrangians known as spectral networks on Riemann surfaces; a possible relation to Conjecture \ref{JoyceConj} is explained at the end of \cite{Haiden_spectral}, Section 1.

The recent work of Chiu and Lin \cite{YLShen_K3FibredSLag} constructs special Lagrangian submanifolds in collapsing Calabi-Yau 3-folds fibered by K3 surfaces, and also discusses a possible relation to Conjecture \ref{JoyceConj} in \cite{YLShen_K3FibredSLag}, Section 6.3.  

Sj\"ostr\"om Dyrefelt and Khalid, \cite{SohaibZak_higherdim} Sections 5.2 and 5.3, study in particular wall-chamber decompositions of the K\"ahler cone according to the solvability of certain dHYM equations with respect to fixed classes. These are very similar to the wall-chamber decompositions of the K\"ahler cone according to the Bridgeland (semi)stability of a given object (see e.g. \cite{BayerMacriToda_Bogomolov}, Proposition 2.2.2). 
\end{rmk}

In Section \ref{MainSec} we provide some background on stability conditions on threefolds and on the results on dHYM connections needed for the proof of Theorem \ref{MainThmIntro}. The theorem is stated more precisely and proved in that Section as Theorem \ref{MainThm}. 

Similarly, Sections \ref{BlpSec} and \ref{RelSec} contain some background on a class of stability conditions on blowups and on relative stability conditions, respectively, together with some preliminary computations on dHYM connections, leading to the proofs of Theorems \ref{BlpThm} and \ref{AdiabaticThm}. 

Section \ref{UnstSec} contains some speculations motivated by the results of Mete \cite{Mete_P3} on $X = \Bl_p\PP^3$. We provide a reasonable candidate for the minimal slope semistable quotient, admitting a weak solution of the special Lagrangian equation, for an unstable Lagrangian section with respect to a (conjectural) Harder-Narasimhan filtration in this example.   
\section{Main result on toric threefolds}\label{MainSec}
In this Section we prove our main result on toric threefolds, Theorem \ref{MainThmIntro}, stated slightly more precisely here as Theorem \ref{MainThm}.

\subsection{Deformed Hermitian Yang-Mills connections}\label{dHYMSec} As for the general results of \cite{J_toricThomasYau}, the proof is based on the Leung-Yau-Zaslow correspondence between special Lagrangian sections and deformed Hermitian Yang-Mills connections.
\begin{thm}[Leung-Yau-Zaslow \cite{LeungYauZaslow}, see also \cite{Chan_survey} Section 3, \cite{CollinsYau_dHYM_momentmap}, Section 9]\label{LYZThm} As in the Introduction, suppose $(X, \omega)$ is a projective toric manifold, endowed with a torus-invariant K\"ahler form $\omega$. 

Then a SYZ Lagrangian section $\cL = \cL(L, h) \subset (\cU_{q}, \omega^{\vee}, \Omega)$ is special Lagrangian, i.e. it satisfies the equation
\begin{equation*}
\Imm e^{-\ii \hat{\theta}} \Omega|_{\cL} = 0
\end{equation*} 
if and only if the Chern connection of $h^{\vee}$ on the dual line bundle $L^{\vee}$ is \emph{deformed Hermitian Yang-Mills}, i.e. it satisfies the equation
\begin{equation*}
\Imm e^{-\ii \hat{\theta}}(\omega - \ii F(h^{\vee}))^n = 0,
\end{equation*}
or equivalently,
\begin{equation*}
\Imm e^{-\ii \hat{\theta}}(\omega + \ii c^{h^{\vee}}_1(L^{\vee}))^n = 0,
\end{equation*}
where $c^{h^{\vee}}_1(L^{\vee})$ denotes the Chern-Weil representative, and 
\begin{equation*}
\int_X(\omega + \ii  c^{h^{\vee}}_1(L^{\vee}))^n \in e^{\ii \hat{\theta}} \R_{>0}.
\end{equation*}
\end{thm}
More generally, fixing an auxiliary $(1, 1)$-class $[\alpha_0]$ on a compact K\"ahler manifold $(X, \omega)$, the dHYM equation seeks a representative $\alpha \in [\alpha_0]$ such that
\begin{equation}\label{dHYM}
\Imm e^{-\ii \hat{\theta}}(\omega + \ii \alpha)^n = 0,
\end{equation}
where
\begin{equation*}
\int_X(\omega + \ii\alpha)^n \in e^{\ii \hat{\theta}} \R_{>0}.
\end{equation*}

Denoting by $\lambda_1, \cdots, \lambda_n$ the eigenvalues of $\omega^{-1} \alpha \in \Gamma(X, \End(T^{1, 0}))$, a direct computation shows that the \emph{constant Langrangian phase equation}
\begin{equation}\label{LagPhaseEqu}
\sum^n_{i = 1} \operatorname{arccot}(\lambda_i) = \varphi \in \R
\end{equation}
implies that the dHYM equation \eqref{dHYM} holds with phase $e^{\ii \hat{\theta}}$ determined by
\begin{equation*}
\varphi = \frac{n}{2}\pi - \hat{\theta} \mod 2\pi.
\end{equation*}

The key existence result for the constant Lagrangian phase equation \eqref{LagPhaseEqu} is the so-called \emph{Nakai-Moishezon} or \emph{dHYM positivity criterion}, due to Chen \cite{GaoChen_Jeq_dHYM}, Datar-Pingali \cite{DatarPingali_dHYM}, Chu-Li-Takahashi \cite{Takahashi_dHYM}, building on the work of Collins, Jacob and Yau (see \cite{CollinsJacobYau, CollinsYau_dHYM_momentmap, JacobYau_special_Lag}) and Collins-Sz\'ekelyhidi \cite{CollinsSzekelyhidi}. 

\begin{thm}[\cite{Takahashi_dHYM}, Corollary 1.5]\label{dHYMThm} Let $X$ be a projective manifold, with fixed classes $\omega_0$, $\alpha_0$ as above. Suppose that there is a lift $\hat{\theta} \in \R$ of the phase $e^{\ii \hat{\theta}}$ such that
\begin{equation*}
\hat{\theta} \in \left(\frac{n-2}{2}\pi, \frac{n}{2}\pi\right).
\end{equation*}
Set 
\begin{equation*}
\varphi := \frac{n}{2}\pi - \hat{\theta} \in (0, \pi).
\end{equation*}

Then there exists a solution of the constant Lagrangian phase equation \eqref{LagPhaseEqu} for a representative $\omega \in [\omega_0]$ if, and only if, for all proper irreducible subvarieties $V \subset X$ we have
\begin{equation}\label{dHYMPosIntro}
\int_V \Rea(\ii \omega_0 + \alpha_0)^{\dim V} - \cot(\varphi)\Imm(\ii \omega_0 +  \alpha_0)^{\dim V} > 0
\end{equation}
The solution is unique.
\end{thm}

The condition \eqref{dHYMPosIntro} is known as \emph{dHYM-positivity} or \emph{dHYM-stability}. Solutions of the constant Lagrangian phase equation \eqref{LagPhaseEqu} with $\varphi \in (0, \pi)$, and of the dHYM equation \eqref{dHYM} induced by these, are called \emph{supercritical}. Note in particular that the existence of supercritical solutions does not depend on the choice of K\"ahler form representing $[\omega_0]$. 
\begin{rmk}\label{ScaleInvarianceRmk} As we mentioned in the Introduction, it follows immediately from Theorems \ref{LYZThm} and \ref{dHYMThm} that the the dHYM positivity condition \eqref{dHYMPosIntro}, and the existence of supercritical solutions, are invariant under the common rescaling $[\omega] \mapsto k [\omega]$, $[\alpha] \mapsto k [\alpha]$ for $k > 0$.
\end{rmk}
The starting point for the proof of Theorem \eqref{MainThm} is the following result, proved in our previous work \cite{J_toricThomasYau}.
\begin{prop}[\cite{J_toricThomasYau}, proof of Theorem 2.5 and Remark 5.5]\label{SVProp} Let $(X, [\omega])$ be a projective toric manifold endowed with a K\"ahler class $[\omega]$ and a line bundle $L$. Suppose that there is a lift $\hat{\theta} \in \R$ of the corresponding phase $e^{\ii \hat{\theta}}$ such that $\hat{\theta} \in \left(\frac{n-2}{2}\pi, \frac{n}{2}\pi\right)$. Let $k > 0$ be sufficiently large, depending only on $(X, [\omega], c_1(L))$. 

There exist finitely many objects $\cS_V \in D^b(X)$, indexed by irreducible toric subvarieties $V\subset X$, with morphisms 
\begin{equation*}
\cS_{V}[-\codim V] \to L^k,
\end{equation*}
such that if we have 
\begin{equation*}
\arg\left( (-1)^{\codim V}\int_{X} e^{-\ii k\omega}\ch(\cS_V)\right) \leq \arg \int_{X} e^{-\ii k\omega}\ch(L^{k}), 
\end{equation*}
then the \emph{dHYM-semipositivity} condition holds for $(X,  [\omega], c_1(L^{\vee}))$, 
\begin{equation}\label{dHYMSemiPos}
\int_V \Rea(\ii \omega + c_1(L^{\vee}))^{\dim V} - \cot(\varphi)\Imm(\ii \omega +  c_1(L^{\vee}))^{\dim V} \geq 0 
\end{equation}
(this condition is also known as \emph{dHYM-semistability}).
\end{prop}
\begin{proof}[Proof (sketch)] We provide a sketch of the proof in the threefold case for the reader's convenience (as this is the main application in the present work). 

When $V \subset X$ is an irreducible toric divisor, the object $\cS_V \in D^b(X)$ associated with $V$ is a two-dimensional torsion sheaf, defined by the short exact sequence 
\begin{equation*}
0 \to L^k \to L^k(k_V V) \to \cS_V \to 0,
\end{equation*} 
where $k \gg k_V \gg 1$ are parameters to be determined. The exact sequence induces a canonical morphism $\cS_V \to L^k[1]$ in $D^b(X)$.

We compute
\begin{align*}
&\int_{X} e^{-\ii k\omega}\ch(\cS_V) = \int_{X} e^{-\ii k\omega} (\ch(L^k(k_V V)) - \ch(L^k))\\
& = \int_{X} e^{-\ii k\omega}\ch(L^k) (\ch(\olo(k_V V)) - 1 )\\
& = k^2 k_V \int_{X} e^{-\ii \omega}\ch(L) \cup c_1(\olo(V)) + O(k k^2_V)\\ 
& = k^2 k_V \int_{V} e^{-\ii \omega}\ch(L) + O(k k^2_V). 
\end{align*}
On the other hand, we note
\begin{equation*}
\int_{X} e^{-\ii k\omega}\ch(L^{k}) = k^3 \int_{X} e^{-\ii \omega}\ch(L).
\end{equation*}
Thus, for a fixed divisor $V$, for $k \gg k_V$, the inequality 
\begin{equation*}
\arg\left( - \int_{X} e^{-\ii k\omega}\ch(\cS_V)\right) \leq \arg \int_{X} e^{-\ii k\omega}\ch(L^{k})  
\end{equation*} 
implies 
\begin{equation*}
\arg\left( - \int_{V} e^{-\ii \omega}\ch(L)\right) \leq \arg \int_{X} e^{-\ii \omega}\ch(L).  
\end{equation*} 
By a direct computation (see e.g. \cite{J_toricThomasYau}, Section 3.2), under our assumptions, the latter inequality is equivalent to dHYM semipositivity \eqref{dHYMSemiPos} with respect to $V$. Since there are only finitely many irreducible toric divisors, we can choose the parameters $k$, $k_V = k_1$ uniformly over all $V$.

Similarly, if $V \subset X$ is a toric curve, we write $V$ as an intersection of toric divisors, $V = V_1 \cap V_2$. Consider the corresponding complexes of line bundles
\begin{equation*}
\cS_{V_2} = [L^k \to L^k(k_2 V_2)],\,\cS_{V_1, V_2} = [L^k(k_1 V_1) \to L^k(k_1 V_1 + k_2 V_2)].
\end{equation*}
The object $\cS_V$ corresponding to $V$ is defined as the cone of the natural morphism   
\begin{equation*}
\cS_{V_2} \to \cS_{V_1, V_2}.
\end{equation*}
In particular there is a canonical morphism $\cS_V \to \cS_{V_2}[1]$, and composing this with $\cS_{V_2}[1] \to L^k[2]$ yields the required morphism $\cS_V \to L^k[2]$.

We compute, using $\Q$-line bundles, 
\begin{align*}
&\int_{X} e^{-\ii k\omega}\ch(\cS_V) = \int_{X} e^{-\ii k\omega}\big(\ch(\cS_{V_1, V_2}) - \ch(\cS_{V_2})\big)\\
& = k^2 k_2 \int_{V_2} e^{-\ii \omega} \ch(L(k^{-1} k_1 V_1)) - k^2 k_2 \int_{V_2} e^{-\ii\omega} \ch(L) + O(k k^2_2)\\    
& = k^2 k_2 \int_{V_2} e^{-\ii \omega} \ch(L)(\ch(\olo(k^{-1} k_1 V_1)) - 1)+ O(k k^2_2)\\
& = k k_1 k_2 \int_{V_2} e^{-\ii \omega} \ch(L) \cup c_1(\olo(V_1)) + O(k^0 k^2_1 k _2) + O(k k^2_2)\\
& = k k_1 k_2 \int_{V_1 \cap V_2} e^{-\ii \omega} \ch(L) + O(k^0 k^2_1 k _2) + O(k k^2_2).  
\end{align*}
Thus, for fixed $k \gg k_1 \gg k_2$, chosen uniformly over the finitely many toric curves $V$, the inequality 
\begin{equation*}
\arg\left( \int_{X} e^{-\ii k\omega}\ch(\cS_V)\right) \leq \arg \int_{X} e^{-\ii k\omega}\ch(L^{k})  
\end{equation*} 
implies 
\begin{equation*}
\arg\left( \int_{V} e^{-\ii \omega}\ch(L)\right) \leq \arg \int_{X} e^{-\ii \omega}\ch(L)   
\end{equation*} 
and by direct computation (see \cite{J_toricThomasYau}, Section 3.2), the latter inequality is equivalent to dHYM semipositivity \eqref{dHYMSemiPos} with respect to curves $V$.  
\end{proof}
\begin{rmk} We note that the objects $\cS_V$ are not unique but depend on certain parameters (in the above proof, in the threefold case, these are denoted by $k, k_1, k_2$). It is shown in \cite{J_toricThomasYau}, proof of Theorem 2.5 and Remark 5.5 that these parameters can be chosen so that, when $X$ is a weak Fano manifold, the central charges
\begin{equation*}
Z(\cS_V) = \int_X e^{-\ii k \omega} \ch(\cS_V)
\end{equation*} 
are close to periods of the mirror holomorphic volume form as $k \to \infty$ (to order $O(k^{-1})$). This is used in the present work to prove Theorem \ref{MainThm} $(iii)$ and other similar statements.
\end{rmk}

Thus we obtain an immediate consequence for dHYM positivity.
\begin{cor}[\cite{J_toricThomasYau}, Remark 5.5]\label{SVCor} Suppose that $[\omega] \in H^{1,1}(X, \R)$ is \emph{generic} with respect to $c_1(L)$ in the sense that it does not lie in the union of the finitely many analytic subvarieties cut out by the equations
\begin{equation*} 
\int_V \Rea(\ii \omega + c_1(L^{\vee}))^{\dim V} - \cot(\varphi)\Imm(\ii \omega +  c_1(L^{\vee}))^{\dim V} = 0,
\end{equation*}
as $V \subset X$ ranges through irreducible toric subvarieties. Then, for sufficiently large $k >0$, the inequalities
\begin{equation}\label{SVIneq}
\arg\left( (-1)^{\codim V}\int_{X} e^{-\ii k\omega}\ch(\cS_V)\right) \leq \arg \int_{X} e^{-\ii k\omega}\ch(L^{k}), 
\end{equation}
imply the dHYM positivity condition for $(X,  [\omega], c_1(L^{\vee}))$, 
\begin{equation*} 
\int_V \Rea(\ii \omega + c_1(L^{\vee}))^{\dim V} - \cot(\varphi)\Imm(\ii \omega +  c_1(L^{\vee}))^{\dim V} > 0.
\end{equation*}
By Theorems \ref{LYZThm} and \ref{dHYMThm}, in this case the SYZ Lagrangian section $\cL(L^k, h)$ is Hamiltonian isotopic to a special SYZ Lagrangian section, i.e. there exists $\tilde{h}$ on the fibres of $L^k$ such that 
\begin{equation*}
\Imm e^{-\ii \hat{\theta}}\Omega|_{\cL(L^k, \tilde{h})} = 0.
\end{equation*}
\end{cor}
In view of Proposition \ref{SVProp} and Corollary \ref{SVCor}, our task is to show that, for sufficiently large $k>0$, the inequalities \ref{SVIneq} are implied by a suitable Bridgeland semistability condition for a shift of $L^k \in D^b(X)$.
\subsection{Some stability conditions on toric threefolds}\label{BMSZSec}
Constructing stability conditions on threefolds is notoriously difficult. In the toric case, under certain restrictive assumptions, a result of Bernardara-Macr\`i-Schmidt-Zhao allows to construct stability conditions with the required central charge for our applications.

Recall that the abelian category $\Coh^{\beta}(X)$ is defined as the tilt of $\Coh(X)$ with respect to the torsion pair
\begin{align*}
&T_{\alpha, \beta} = \{E \in \Coh(X)\!: \mu_{\alpha,\beta; \min}(E) > 0\},\\
&F_{\alpha, \beta} = \{E \in \Coh(X)\!: \mu_{\alpha,\beta; \max}(E) \leq 0\},
\end{align*} 
where $\mu_{\omega,\beta; \min}$, $\mu_{\alpha,\beta; \max}$ denote the minimal and maximal slopes of the Harder-Narasimhan factors of $E$ with respect to the twisted Mumford-Takemoto slope function, given by
\begin{equation*}
\mu_{\omega, \beta}(E) = \frac{\omega^2 \cdot \ch^{\beta}_1(E)}{\omega^3 \cdot \ch^{\beta}_0(E)}, 
\end{equation*}
writing $\ch^{\beta} := \ch \cdot e^{-\beta\omega}$ for the twisted Chern character. That is, we have, by definition, 
\begin{equation*}
\Coh^{\beta}(X) := \langle F_{\alpha, \beta}[1], T_{\alpha, \beta} \rangle,
\end{equation*}
where the right hand side denotes the extension closure of the set of objects $\{ F_{\alpha, \beta}[1], T_{\alpha, \beta} \}$ in $D^b(X)$ (see e.g. \cite{BayerMacriToda_Bogomolov}, Section 3.1).

Introduce a slope function $\nu^{\omega}_{\alpha, \beta}$ on the abelian category $\Coh^{\beta}(X)$ by
\begin{equation*}
\nu^{\omega}_{\alpha, \beta}(E) = \frac{\omega \cdot \ch^{\beta}_2(E) - \frac{\alpha^2}{2} \omega^3\cdot \ch^{\beta}_0(E)}{\omega^2 \cdot \ch^{\beta}_1(E)}.
\end{equation*}

Following \cite{BayerMacriToda_Bogomolov}, Definition 3.2.3, an object $E \in \Coh^{\beta}(X)$ is called $\nu^{\omega}_{\alpha, \beta}$-\emph{(semi)stable} if for any nonzero proper subobject $F \subset E$ we have
\begin{equation*}
\nu^{\omega}_{\alpha, \beta}(E) < (\leq) \nu^{\omega}_{\alpha, \beta}(E/F).
\end{equation*}
\begin{thm}[\cite{BernardaraMacri_threefolds}, Theorem 5.1]\label{ToricBStabThm} Let $X$ be a smooth projective toric threefold (not necessarily Fano), with an ample divisor $H$ such that, for all effective divisors $D$, we have
\begin{enumerate}
\item[$(i)$] $H \cdot D^2 \geq 0$,  
\item[$(ii)$] if $H \cdot D^2 = 0$, then $D$ lies on an extremal ray of the effective cone.
\end{enumerate}
Then, any $\nu_{\alpha, \beta}$-stable object $E \in \Coh^{\beta}(X)$ with $\nu_{\alpha, \beta}(E) = 0$ satisfies the Bogomolov-Gieseker type inequality
\begin{equation}\label{BGineq}
\ch^{\beta}_3(E) \leq \frac{\alpha^2}{6} H^2 \cdot \ch^{\beta}_1(E).
\end{equation}
\end{thm}
\begin{rmk} As the authors of \cite{BernardaraMacri_threefolds} observe, the result is invariant under scaling $H$, so we may replace $H$ with an ample rational class $\omega$ in our applications.
\end{rmk}
\begin{rmk} The authors of \cite{BernardaraMacri_threefolds} expect that condition $(ii)$ is not necessary. They observe that it is automatically satisfied for projective bundles over $\PP^1$. On the other hand, condition $(i)$ is quite restrictive and necessary. Indeed, a well-known counterexample to the Bogomolov-Gieseker type inequality \eqref{BGineq} (discovered in \cite{Schmidt_Bogomolov}) is given by $X = \Bl_p \PP^3$, $H = -\frac{K_X}{2}$, $E = \olo_{\PP^3}(1)$.
\end{rmk}

The general theory developed by Bayer-Macr\`i-Toda \cite{BayerMacriToda_Bogomolov} allows to construct stability conditions on any smooth projective threefold $X$, with Douglas type central charges, starting from the Bogomolov-Gieseker type inequality \eqref{BGineq}. 

Introduce the subcategories $\cT_{\alpha, \beta}$, $\cF_{\alpha, \beta}$ of $\Coh^{\beta}(X)$ given by 
\begin{align*}
&\cT_{\alpha, \beta} = \{E \in \Coh^{\beta}(X)\!: \nu_{\alpha,\beta; \min}(E) > 0\},\\
&\cF_{\alpha, \beta} = \{E \in \Coh^{\beta}(X)\!: \nu_{\alpha,\beta; \max}(E) \leq 0\},
\end{align*} 
where $\nu_{\omega,\beta; \min}$, $\nu_{\alpha,\beta; \max}$ denote the minimal and maximal slopes of the Harder-Narasimhan factors of $E$ with respect to the slope function $\nu_{\alpha,\beta}$ (these are well defined by \cite{BayerMacriToda_Bogomolov}, Lemma 3.2.4). Then, according to  \cite{BayerMacriToda_Bogomolov}, Section 3.2, $(\cT_{\alpha, \beta}, \cF_{\alpha, \beta})$ is a torsion pair on $\Coh^{\beta}(X)$, and we denote the corresponding tilt by 
\begin{equation*}
\cA_{\alpha, \beta} := \langle \cF_{\alpha, \beta}[1], \cT_{\alpha, \beta} \rangle.
\end{equation*}
This is the heart of a bounded t-structure on $D^b(X)$ and so in particular it is an abelian category. Let us also introduce a central charge on $D^b(X)$ given by 
\begin{align*}
Z_{\alpha, \beta}(E) := \left(-\ch^{\beta}_3(E) + \frac{3\alpha^2}{2} \omega^2 \cdot \ch^{\beta}_1(E)\right) + \ii \left(\omega\cdot \ch^{\beta}_2(E) - \frac{\alpha^2}{2}\omega^3 \cdot \ch^{\beta}_0(E)\right).
\end{align*} 
\begin{rmk} We will write $\nu^{\omega}_{\alpha,\beta}$, $\cF^{\omega}_{\alpha, \beta}$, $\cT^{\omega}_{\alpha, \beta}$, $Z^{\omega}_{\alpha, \beta}$ when it is necessary to emphasise the dependence on the polarisation $\omega$. 
\end{rmk}
One of the key results of \cite{BayerMacriToda_Bogomolov} can be summarised as follows.
\begin{thm}[\cite{BayerMacriToda_Bogomolov}, Corollary 5.2.4]\label{BMTThm} If the Bogomolov-Gieseker type inequality \eqref{BGineq} holds for $\nu_{\alpha, \beta}$-stable objects $E \in \Coh^{\beta}(X)$ with $\nu_{\alpha, \beta}(E) = 0$ on the smooth projective threefold $X$, then the pair
\begin{equation*}
\sigma_{\alpha, \beta} := \left(\cA_{\alpha, \beta}, Z_{\alpha, \beta}\right)
\end{equation*}
defines a Bridgeland stability condition on $D^b(X)$.
\end{thm}
\begin{rmk} When $\frac{\alpha^2}{2} = \frac{1}{6}$, we obtain the central charge
\begin{equation*}
Z_{\omega, \beta}(E) = -\int_X e^{-\ii \omega}e^{-\beta \omega}\ch(E).
\end{equation*}
We denote the corresponding stability condition by 
\begin{equation*}
\sigma_{\omega, \beta} := \left(\cA_{\omega, \beta}, Z_{\omega, \beta}\right) =  \left(\langle \cF_{\omega, \beta}[1], \cT_{\omega, \beta} \rangle , Z_{\omega, \beta}\right)
\end{equation*} 
(when it is well defined). 
\end{rmk}
Thus, in the situation of Theorem \ref{ToricBStabThm}, we obtain stability conditions $\sigma_{\alpha, \beta}$, and in particular $\sigma_{\omega, \beta}$ (for $\frac{\alpha^2}{2} = \frac{1}{6}$) on the toric threefold $X$. 
\subsection{Auxiliary results and Theorem \ref{MainThm}}
Let us define the phase $e^{\ii \hat{\theta}_{\beta}}$ by 
\begin{equation*}
\int_X (\omega + \ii (c_1(L^{\vee}) + \beta \omega))^3 \in e^{\ii \hat{\theta}_{\beta}}\R_{>0}.
\end{equation*}
By direct computation, the real part of $\int_X (\omega + \ii (c_1(L^{\vee}) + \beta \omega))^3$
is given by 
\begin{equation*}
-\omega \cdot (c_1(L) - \beta \omega)^2 + \frac{\omega^3}{3} = R \cos(\hat{\theta}_{\beta}),\,R>0.
\end{equation*}
If $\cos(\hat{\theta}_{\beta}) < 0$, we can find a lift $\hat{\theta}_{\beta} \in \R$ lying in the supercritical interval $\hat{\theta}_{\beta} \in (\frac{\pi}{2}, \frac{3 \pi}{2})$. 
\begin{prop}\label{DivisorsProp} Let $(X, \omega)$ satisfy the assumptions of Theorem \ref{ToricBStabThm}. Suppose $L$ is a line bundle on $X$ with $\mu_{\omega,\beta}(L) < 0$ and satisfying the condition
\begin{equation*}
\cos(\hat{\theta}_{\beta}) < 0 \iff \omega \cdot (c_1(L) - \beta \omega)^2 - \frac{\omega^3}{3} > 0. 
\end{equation*}    
Then, for a uniform choice of the parameters $k \gg k_1 \gg 1$ appearing in the proof of Proposition \ref{SVProp}, for all irreducible toric \emph{divisors} $V \subset X$, there is an inclusion 
\begin{equation*}
\cS_V[1] \subset L^k[2]  
\end{equation*}
in the heart $\cA^{k\omega}_{\alpha, k\beta}$ of the Bridgeland stability condition 
\begin{equation*}
\sigma^{k\omega}_{\alpha, k\beta} = \left(\cA^{k\omega}_{\alpha, k\beta}, Z^{k\omega}_{\alpha, k\beta}\right)
\end{equation*}
for $0 < \alpha \leq \frac{1}{\sqrt{3}}$.
\end{prop}
\begin{proof} Let $k > 0$ denote a large parameter to be determined. We assume that $L$ has negative twisted slope 
\begin{equation*}
\mu_{\omega, \beta}(L) < 0. 
\end{equation*}
So, by definition, we have $L^k[1] \in \Coh^{k\beta}(X)$ for all $k > 0$.

Fix a toric divisor  $V \subset X$. Note that, for $k \gg k_V \gg 1$, 
\begin{equation*}
\mu_{k\omega, k\beta}(L^k(k_V V)) =  \mu_{\omega, \beta}(L) + k^{-1}\mu_{\omega}(\olo(k_V V)) < 0,
\end{equation*} 
so we also have $L^k(k_V V)[1] \in \Coh^{k\beta}(X)$.

As we recalled, the object $\cS_V \in D^b(X)$ associated with $V$ is a two-dimensional torsion sheaf, defined by the short exact sequence 
\begin{equation*}
0 \to L^k \to L^k(k_V V) \to \cS_V \to 0,
\end{equation*} 
where $k \gg k_V \gg 1$. This implies that $\cS_V$ fits in the exact triangle  
\begin{equation*}
\cS_V \to L^k[1] \to L^k(k_V V)[1] \to \cS_V[1]. 
\end{equation*} 
Since $\cS_V$ is torsion, we also have $\cS_V \in \Coh^{k\beta}(X)$. Thus, the above exact triangle is induced by a short exact sequence in $\Coh^{k\beta}(X)$,
\begin{equation*}
\cS_V \to L^k[1] \to L^k(k_V V)[1].   
\end{equation*} 
Therefore, subobjects of $\cS_V$ in $\Coh^{k\beta}(X)$ are identified with subobjects of $L^k[1]$ in $\Coh^{k\beta}(X)$.  
Suppose we fix our conditions so that 
\begin{equation*}
L^k[1] \in \cF^{k\omega}_{\alpha, k\beta}. 
\end{equation*}
Then, for all subobjects $S \subset L^k[1]$ in $\Coh^{k\beta}(X)$, we have $\nu^{k\omega}_{\alpha, k\beta}(S) \leq 0$. So the same is true for subobjects $S \subset \cS_V$, and, by an equivalent characterisation of $\cF^{k\omega}_{\alpha, k\beta}$ (see e.g. \cite{Schmidt_blowups}, Section 2), $\cS_V$ also lies in $\cF^{k\omega}_{\alpha, k\beta}$. By definition, it follows that $L^k[2]$, $\cS_V[1]$ lie in the heart $\cA^{k\omega}_{\alpha, k\beta}$ of the stability condition $\sigma^{k\omega}_{\alpha, k\beta}$.

Let us recall a key fact.
\begin{lemma}[\cite{BayerMacriStellari_3folds}, Corollary 3.11 $(b)$] Suppose that, for all effective divisors $D$, we have $\omega\cdot D^2 \geq 0$. Then, if $L$ is any line bundle, either $L$ or $L[1]$ must be a $\nu_{\alpha, \beta}$-stable object of $\Coh^{\beta}(X)$, for all $\beta$ and all $\alpha > 0$.
\end{lemma}
Applying this in our case, we see that $L^k[1]$ is in fact a $\nu^{k\omega}_{\alpha, k\beta}$-stable object of $\Coh^{k\beta}(X)$, so we have
\begin{equation*}
\nu^{k\omega}_{\alpha, k\beta; \min}(L^k[1]) = \nu^{k\omega}_{\alpha, k\beta}(L^k[1]). 
\end{equation*}
It follows that
\begin{align*}
& L^k[1] \in \cF^{k\omega}_{\alpha, k\beta} \iff \nu^{k\omega}_{\alpha, k\beta}(L^k[1]) \leq 0.
\end{align*} 
By definition, 
\begin{equation*}
\nu^{\omega}_{\alpha, \beta}(E) = \frac{\omega \cdot \ch^{\beta}_2(E) - \frac{\alpha^2}{2} \omega^3\cdot \ch^{\beta}_0(E)}{\omega^2 \cdot \ch^{\beta}_1(E)},
\end{equation*}
so we have
\begin{align*}
\nu^{k\omega}_{\alpha, k\beta}(L^k[1])  =  \nu^{k\omega}_{\alpha, k\beta}(L^k) = \nu^{\omega}_{\alpha, \beta}(L) = \frac{\omega \cdot \frac{1}{2}(c_1(L) - \beta \omega)^2 - \frac{\alpha^2}{2}\omega^3}{\omega^2 \cdot (c_1(L) - \beta\omega)}. 
\end{align*}
Since we are assuming $\mu_{\omega, \beta}(L) < 0$, the denominator is negative. By our condition $\cos(\hat{\theta}_{\beta}) < 0$, the numerator is positive for $0 < \alpha < \frac{1}{\sqrt{3}}$. Thus, with our assumptions, we have $\nu^{k\omega}_{\alpha, k\beta}(L^k[1]) < 0$ and so $L^k[1],\,S_V \in \cF^{k\omega}_{\alpha, k\beta}$, which implies
\begin{equation*}
L^k[2],\,S_V[1] \in \cA^{k\omega}_{\alpha, k\beta}
\end{equation*}
as required. Similarly, we compute
\begin{align*}
&\nu^{k\omega}_{\alpha, k\beta}(L^k(k_V V)[1]) = \nu^{k\omega}_{\alpha, k\beta}(L^k(k_V V))\\
&=\frac{k \omega \cdot \left(\frac{1}{2}(c_1(L^k) + c_1(\olo(k_V V)))^2 - k \beta \omega \cdot (c_1(L^k) + c_1(\olo(k_V V)) + k^2 \beta^2 \frac{\omega^2}{2}\right)}{k \omega^2 \cdot (c_1(L^k)-k\beta\omega + c_1(\olo(k_V V)))}\\
&-\frac{k^3\frac{\alpha^2}{2}\omega^3}{k \omega^2 \cdot (c_1(L^k)-k\beta\omega + c_1(\olo(k_V V)))}. 
\end{align*}
Since we are assuming $\mu_{\omega, \beta}(L) < 0$, the denominator is negative for  $k \gg k_V \gg 1$, while the leading order term in $k$ of the numerator is positive for $0 < \alpha < \frac{1}{\sqrt{3}}$, by our condition $\cos(\hat{\theta}_{\beta}) < 0$. 

Thus, for $k \gg k_V \gg 1$, depending on $V$, with our assumptions, we have 
\begin{equation*}
\nu^{k\omega}_{\alpha, k\beta}(L^k(k_V V)[1]) < 0,
\end{equation*}
which yields $L^k(k_V V)[1] \in \cF^{k\omega}_{\alpha, k\beta}$ and so $L^k(k_V V)[2] \in \cF^{k\omega}_{\alpha, k\beta}[1] \subset \cA^{k\omega}_{\alpha, k\beta}$. It follows that the exact triangle 
\begin{equation*}
\cS_V[1] \to L^k[2] \to L^k(k_V V)[2] \to \cS_V[2]   
\end{equation*}  
is induced by a short exact sequence
\begin{equation*}
\cS_V[1] \to L^k[2] \to L^k(k_V V)[2]     
\end{equation*}  
in $\cA^{k\omega}_{\alpha, k\beta}$. Finally note that we can choose $k \gg k_V \gg 1$ uniformly in $V$.
\end{proof}
\begin{prop}\label{codim2Prop} Let $(X, \omega)$ satisfy the assumptions of Theorem \ref{ToricBStabThm}. Suppose $L$ is a line bundle on $X$ with $\mu_{\omega,\beta}(L) < 0$ and satisfying the condition
\begin{equation*}
\cos(\hat{\theta}_{\beta}) < 0 \iff \omega \cdot (c_1(L) - \beta \omega)^2 - \frac{\omega^3}{3} > 0. 
\end{equation*}    
Then, for a uniform choice of the parameters $k, k_1, k_2$ appearing in the proof of Proposition \ref{SVProp}, for all toric \emph{curves} $V \subset X$, there is an inclusion 
\begin{equation*}
\cS_V \subset L^k[2] 
\end{equation*}
in the heart $\cA^{k\omega}_{\alpha, k\beta}$ of the Bridgeland stability condition 
\begin{equation*}
\sigma^{k\omega}_{\alpha, k\beta} = \left(\cA^{k\omega}_{\alpha, k\beta}, Z^{k\omega}_{\alpha, k\beta}\right)
\end{equation*}
for $0 < \alpha \leq \frac{1}{\sqrt{3}}$.
\end{prop}
\begin{proof} Let $V \subset X$ be a toric $1$-dimensional submanifold. We write $V$ as an intersection of toric divisors, $V = V_1 \cap V_2$. Consider the corresponding complexes of line bundles
\begin{equation*}
\cS_{V_2} = [L^k \to L^k(k_2 V_2)],\,\cS_{V_1, V_2} = [L^k(k_1 V_1) \to L^k(k_1 V_1 + k_2 V_2)].
\end{equation*}
The cohomology of $\cS_{V_2}$, $\cS_{V_1, V_2}$ is supported in degree $0$ and given by torsion sheaves, so we have 
$\cS_{V_2},\, \cS_{V_1, V_2} \in \Coh^{\beta}(X)$.

By construction, the object $\cS_V$ corresponding to $V$ is given by   
\begin{equation*}
\cS_V \cong [\cS_{V_2} \to \cS_{V_1, V_2}].
\end{equation*} 
So $\cS_V$ is actually represented by a complex in $\Coh^{\beta}(X)$. Its cohomology is supported in degree $0$ and given by the cokernel of the morphism of sheaves
\begin{equation*}
L^k(k_1 V_1)|_{k_1 V_1} \to L^k(k_1 V_1 + k_2 V_2)|_{k_1 V_1}.
\end{equation*} 
Therefore, it is a torsion sheaf of dimension $1$, i.e. an element of $\Coh^{\leq 1}(X)$. It follows that $\cS_V \cong [\cS_{V_2} \to \cS_{V_1, V_2}]$ satisfies
\begin{equation*}
\nu^{k\omega}_{\alpha, k\beta}(\cS_V) = + \infty
\end{equation*}
(see \cite{BayerMacriToda_Bogomolov}, Remark 3.2.2) and so $\cS_V \in \cT^{k\omega}_{\alpha, k\beta} \subset \cA^{k\omega}_{\alpha, k\beta}$.

The exact triangle defining $\cS_V$
\begin{equation*}
\cS_{V_2} \to \cS_{V_1, V_2} \to \cS_V \to \cS_{V_2}[1]
\end{equation*}
induces an exact triangle
\begin{equation*}
\cS_V \to \cS_{V_2}[1] \to \cS_{V_1, V_2}[1] \to \cS_V[1].
\end{equation*}
By the same argument as in the proof of Proposition \ref{DivisorsProp}, we have 
\begin{equation*}
\cS_{V_2}[1],\,\cS_{V_1, V_2}[1] \in \cF^{k\omega}_{\alpha, k\beta}[1] \subset \cA^{k\omega}_{\alpha, k\beta},
\end{equation*} 
so the latter triangle must be induced by an exact sequence 
\begin{equation*}
\cS_V \to \cS_{V_2}[1] \to \cS_{V_1, V_2}[1] 
\end{equation*}
in $\cA^{k\omega}_{\alpha, k\beta}$, and there is an injection $\cS_V \subset \cS_{V_2}[1]$ in $\cA^{k\omega}_{\alpha, k\beta}$. Composing this with the injection $\cS_{V_2}[1] \subset L^k[2]$ in $\cA^{k\omega}_{\alpha, k\beta}$ given by Proposition \ref{DivisorsProp} proves the claim.
\end{proof}
\begin{cor}\label{BStabdHYMCor} Let $(X, \omega)$ satisfy the assumptions of Theorem \ref{ToricBStabThm}. Fix a line bundle  $L$ on $X$ with $\mu_{\omega, \beta}(L) < 0$ and satisfying the condition
\begin{equation*}
\cos(\hat{\theta}_{\beta}) < 0 \iff \omega \cdot (c_1(L) - \beta \omega)^2 - \frac{\omega^3}{3} > 0. 
\end{equation*}
Suppose that 
\begin{enumerate}
\item[$(i)$] $L^k[2]$ is \emph{semistable} with respect to the Bridgeland stability condition $\sigma^{k\omega}_{\alpha, k\beta}$ for sufficiently large $k > 0$ and $0 < \alpha \leq \frac{1}{\sqrt{3}}$ sufficiently close to $\frac{1}{\sqrt{3}}$,
\item[$(ii)$] the K\"ahler class $[\omega]$ is \emph{generic} in the sense that it lies in the complement of the union of analytic subvarieties    
\begin{equation*} 
\int_V \Rea(\ii \omega + \beta \omega + c_1(L^{\vee}))^{\dim V} - \cot(\varphi_{\beta})\Imm(\ii \omega + \beta\omega+ c_1(L^{\vee}))^{\dim V} =  0
\end{equation*}
(as $V$ ranges through irreducible toric subvarieties).
\end{enumerate}

Then, the dHYM equation with $B$-field given by $\beta \omega$, namely 
\begin{equation*}
\Imm e^{-\ii \hat{\theta}_{\beta}} (\omega + \ii (c^{h^{\vee}}_1(L^{\vee}) + \beta \omega))^n = 0
\end{equation*}
is solvable.
\end{cor}
\begin{proof} By our condition 
\begin{equation*}
\cos(\hat{\theta}_{\beta}) < 0 \iff \omega \cdot (c_1(L) - \beta \omega)^2 - \frac{\omega^3}{3} > 0
\end{equation*}
it is possible to fix a lift $\hat{\theta}_{\beta}$ of the topological phase $e^{\ii \hat{\theta}_{\beta}}$ lying in the supercritical interval $\hat{\theta}_{\beta} \in \left(\frac{\pi}{2}, \frac{3}{2}\pi\right)$. So the dHYM equation with $B$-field
\begin{equation*}
\Imm e^{-\ii \hat{\theta}_{\beta}} (\omega + \ii (c^{h^{\vee}}_1(L^{\vee})+\beta\omega))^n = 0
\end{equation*}
is implied by the supercritical constant Lagrangian phase equation
\begin{equation*} 
\sum^n_{i = 1} \operatorname{arccot}(\lambda_i) = \frac{3}{2}\pi - \hat{\theta}_{\beta},
\end{equation*}
where $\lambda_i$ denote the eigenvalues of $\omega^{-1} (c^{h^{\vee}}_1(L^{\vee}) + \beta \omega)$.

By Proposition \ref{DivisorsProp}, if $L^k[2]$ is semistable with respect to $\sigma^{k\omega}_{\alpha, k\beta}$, and $V \subset X$ is an irreducible toric divisor, we must have 
\begin{equation*}
\arg Z^{k\omega}_{\alpha, k\beta}(\cS_V[1]) \leq \arg Z^{k\omega}_{\alpha, k\beta}(L^k[2]). 
\end{equation*}
For all $0 < \alpha \leq \frac{1}{\sqrt{3}}$ sufficiently close to $\frac{1}{\sqrt{3}}$, uniformly in $V$, this is equivalent to the inequality
\begin{equation*} 
\arg\left( - \int_{X} e^{-\ii k\omega - \beta k\omega}\ch(\cS_V)\right) \leq \arg \int_{X} e^{-\ii k\omega- \beta k\omega}\ch(L^{k}). 
\end{equation*}
Similarly, by Proposition \ref{codim2Prop}, if $V \subset X$ is a toric curve, we must have 
\begin{equation*}
\arg Z^{k\omega}_{\alpha, k\beta}(\cS_V) \leq \arg Z^{k\omega}_{\alpha, k\beta}(L^k[2]),  
\end{equation*}
or equivalently
\begin{equation*} 
\arg\left(\int_{X} e^{-\ii k\omega- \beta k\omega}\ch(\cS_V)\right) \leq \arg \int_{X} e^{-\ii k\omega- \beta k\omega}\ch(L^{k}). 
\end{equation*}
Thus, the inequalities \eqref{SVIneq} hold for all irreducible toric subvarieties $V \subset X$. As in Proposition \ref{SVProp}, the above phase inequalities for irreducible toric subvarieties $V\subset X$ imply the dHYM-semipositivity condition  
\begin{equation*} 
\int_V \Rea(\ii \omega + \beta\omega + c_1(L^{\vee}))^{\dim V} - \cot(\varphi)\Imm(\ii \omega + \beta\omega + c_1(L^{\vee}))^{\dim V} \geq 0,
\end{equation*}
and, as in Corollary \ref{SVCor}, if $[\omega]$ is generic in the sense of $(ii)$ in the statement, strict positivity must hold. The claim now follows from Theorem \ref{dHYMThm}. 
\end{proof}

Toric homological mirror symmetry gives an equivalence  
\begin{equation*}
D^b(X) \cong \FS(\cY_{q_k}).
\end{equation*}
Through this equivalence, each stability condition $\sigma^{k\omega}_{\alpha, k\beta} = \left(\cA^{k\omega}_{\alpha, k\beta}, Z^{k\omega}_{\alpha, k\beta}\right)$ on $D^b(X)$ induces a stability condition on $\FS(\cY_{q_k})$. 

We denote by $\sigma^{\vee} = \left(Z^{\vee}, \cA^{\vee}\right)$ the stability condition on $\FS(\cY_{q_k})$ corresponding to $\sigma^{k\omega}_{1/\sqrt{3}, 0}$ for fixed large $k > 0$. The central charge satisfies
\begin{equation*}
Z^{\vee}(\cL(L^k, h)) := Z(L^k) = \int_X e^{-\ii k\omega} \ch(L^k). 
\end{equation*}
Through homological mirror symmetry, the objects $\cS_V \in D^b(X)$ appearing in Propositions \ref{DivisorsProp}, \ref{codim2Prop} correspond to objects $\cL_V \in \FS(\cY_{q_k})$. When $V \subset X$ is a toric divisor, there is an inclusion  
\begin{equation*}
\cL_V[1] \subset \cL(L^k, h)[2] 
\end{equation*}
in the heart $\cA^{\vee} \subset \FS(\cY_{q_k})$. Similarly, when $V \subset X$ is a toric curve, there is an inclusion
\begin{equation*}
\cL_V \subset \cL(L^k, h)[2]. 
\end{equation*}
\begin{thm}\label{MainThm} Let $(X, \omega)$ satisfy the assumptions of Theorem \ref{ToricBStabThm}. Fix a line bundle  $L$ on $X$ with $\mu_{\omega}(L) < 0$ and satisfying the condition
\begin{equation*}
\cos(\hat{\theta}) < 0 \iff \omega \cdot (c_1(L))^2 - \frac{\omega^3}{3} > 0. 
\end{equation*} 
Define the graded Lagrangian section
\begin{equation*}
\tilde{\cL} = \cL(L^k, h)[2] 
\end{equation*}
as a shift of the SYZ transform $\cL(L^k, h) \subset \cU_{q_k} \subset \cY_{q_k}$ for $k > 0$ sufficiently large (depending only on $(X, \omega)$ and $L$). Then 
\begin{enumerate}
\item[$(i)$] there are inclusions $\cL_V[\dim V - 1] \subset \tilde{\cL}$ in the heart $\cA^{\vee}$ of the Bridgeland stability condition $\sigma^{\vee}$ on $\FS(\cY_{q_k})$.  
\item[$(ii)$] If $[\omega] \in H^{1,1}(X, \R)$ is \emph{generic} in the sense of Corollary \ref{BStabdHYMCor}, $(ii)$ and $\tilde{\cL}$ is \emph{semistable} with respect to the stability condition $\sigma^{\vee}$ on $\FS(\cY_{q_k})$, then the inequalities 
\begin{equation*}
\arg Z^{\vee}\big(\cL_V[\dim V - 1]\big) \leq \arg Z^{\vee}\big(\tilde{\cL}\big)
\end{equation*}
hold, and $\tilde{\cL}$ is isomorphic in $\FS(\cY_{q_k})$ to a special Lagrangian SYZ section.
\item[$(iii)$] If $X$ is weak Fano, the central charges of $\tilde{\cL}$ and $\cL_V$ are given by periods of the holomorphic volume form, up to a small correction term 
\begin{align*}
&Z^{\vee}(\tilde{\cL}) = \frac{1}{(2\pi \ii)^{n}} \int_{[\tilde{\cL}]} e^{-W(k \omega_0)/z} \Omega_0\,(1 + O(k^{-1})),\\
&Z^{\vee}(\cL_V) = \frac{1}{(2\pi \ii)^{n}} \int_{[\cL_V]} e^{-W(k \omega_0)/z} \Omega_0\,(1 + O(k^{-1})).
\end{align*}
\item[$(iv)$] When $X$ is not necessarily weak Fano, there exist a Landau-Ginzburg potential $W_{q_k}$, \emph{complex} cycles $\Gamma_{\cL}$, $\Gamma_{\cL_V}$ with
\begin{equation*}
[\Gamma_{\cL}],\,[\Gamma_{\cL_V}] \in H_n(\cY_{q_k}, \{\Rea(W_{q_k}) \gg 0\}; \Z) \otimes \C,
\end{equation*}
and a holomorphic volume form $\Omega^{(k)}$, such that  
\begin{align*} 
& Z^{\vee}(\tilde{\cL}) = \frac{1}{(2\pi \ii)^{n}}\int_{\Gamma_{\cL}} e^{-W(k\omega_0)/z}\Omega^{(k)} (1 + O(k^{-1})),\\  
& Z^{\vee}(\cL_V) = \frac{1}{(2\pi \ii)^{n}}\int_{\Gamma_{\cL_V}} e^{-W(k\omega_0)/z}\Omega^{(k)} (1 + O(k^{-1})),
\end{align*}
where the integrals are understood in the sense of asymptotic expansions for $z \to 0^+$ (see e.g.  \cite{CoatesCortiIritani_hodge}, Section 6.20).
\end{enumerate}
\end{thm}
\begin{proof} The inclusions $\cL_V[\dim V - 1] \subset \tilde{\cL}$ in $\cA^{\vee}$ claimed in $(i)$ have already been observed. If $\tilde{\cL}$ is $\sigma^{\vee}$-semistable they imply the inequalities
\begin{equation*}
\arg Z^{\vee}\big(\cL_V[\dim V - 1]\big) \leq \arg Z^{\vee}\big(\tilde{\cL}\big).
\end{equation*}
If $[\omega] \in H^{1,1}(X, \R)$ is \emph{generic} in the sense of Corollary \ref{BStabdHYMCor}, $(ii)$, these inequalities imply that the SYZ Lagrangian section $\cL(L^k, h)$ is Hamiltonian isotopic to a special Lagrangian, by Corollary \ref{SVCor}. The isomorphism class in $\FS(\cY_{q_k})$ of $\tilde{\cL} = \cL(L^k, h)[2]$ is represented by the same underlying Lagrangian as $\cL(L^k, h)$, with a shift of $2\pi$ in the grading, see \cite{YangLi_ThomasYau}, footnote 1 and Remark 2.6. Thus, $\tilde{\cL}$ is isomorphic in $\FS(\cY_{q_k})$ to a special Lagrangian SYZ section. This proves our claim $(ii)$.

The asymptotic expansions for the central charges in the weak Fano case claimed in $(iii)$ are proved in \cite{J_toricThomasYau}, Section 5.2 (proof of Theorem 2.5) as an application of Iritani's Gamma theorem \cite{Iritani_gamma}.

The weaker statement $(iv)$ concerning the general toric case is proved in \cite{J_toricThomasYau}, Theorem 2.20.
\end{proof}
\begin{rmk}\label{GammaRmk} It follows from the Gamma theorem that, when $X$ is weak Fano, the difference between $Z^{\vee}(\tilde{\cL})$ and a period of $\Omega$ is measured by the gamma class $\widehat{\Gamma}_X$ and the Givental function $I_X(q_k, - z)$, see \cite{J_toricThomasYau}, Section 5 for more details.
\end{rmk}
\section{Blowups}\label{BlpSec}
In this Section we will apply the results of Martinez-Schmidt on stability conditions on blowups in order to obtain more examples for which the conclusions of Theorem \ref{MainThmIntro} hold. The latter result is proved as Theorem \ref{BlpThm} at the end of the Section.

Let $(X, \omega)$ be a smooth projective toric threefold with a Hodge torus-invariant K\"ahler form $\omega$ satisfying the conditions of Theorem \ref{ToricBStabThm}, i.e. such that, for all effective divisors $D$, we have
\begin{enumerate}
\item[$(a)$] $\omega \cdot D^2 \geq 0$,
\item[$(b)$] if $\omega \cdot D^2 = 0$, then $D$ lies on an extremal ray of the effective cone.
\end{enumerate}
Fix a line bundle  $L$ on $X$ with $\mu_{\omega}(L) < 0$, satisfying the condition
\begin{equation*}
\omega \cdot (c_1(L))^2 - \frac{\omega^3}{3} > 0. 
\end{equation*}
Let $\pi\!:\tilde{X} := \Bl_{p}X \to X$ denote the blowup at a torus-fixed point, with exceptional divisor $E$. 

Following \cite{Schmidt_blowups}, Section 4, we define divisor and curve classes on $\tilde{X}$ by
\begin{equation*}
\tilde{\omega} = \pi^* \omega,\,\tilde{B} = 2E + \beta \tilde{\omega},\,\tilde{\Gamma} = -\frac{E^2}{6},
\end{equation*}  
where $\beta \in \R$. For $\alpha > 0$, we define a central charge
\begin{equation*}
Z^{\tilde{\Gamma}}_{\alpha, \beta} = \left(-\ch^{\tilde{B}}_3 + \frac{3\alpha^2}{2} \tilde{\omega}^2\cdot\ch^{\tilde{B}}_1 + \tilde{\Gamma}\cdot \ch^{\tilde{B}}_1\right) + \ii \left(\tilde{\omega}\cdot \ch^{\tilde{B}}_2 - \frac{\alpha^2}{2}\tilde{\omega}^3\cdot\ch^{\tilde{B}}_0\right).
\end{equation*} 
There is a sheaf of $\olo_X$-algebras  
\begin{equation*}
\cB = \pi_{*} \shEnd(\cE),\, \cE:= \olo_{\tilde{X}} \oplus \olo_{\tilde{X}}(E) \oplus \olo_{\tilde{X}}(2E) 
\end{equation*}
which yields an equivalence of $D^b(\tilde{X})$ with the derived category of the category of finitely generated $\cB$-modules, 
\begin{equation*}
\Phi\!: D^b(\tilde{X})  \xrightarrow{\cong} D^b(X, \cB),\,F \mapsto R\pi_* R\shHom(\cE, F) 
\end{equation*} 
see \cite{Schmidt_blowups}, Theorem 4.1. Let 
\begin{equation*}
j\!:  D^b(X, \cB) \to D^b(X)
\end{equation*} 
denote the forgetful functor. One can define a forgetful Chern character on $D^b(X, \cB)$ by
\begin{equation*}
\ch(F) = \ch(j(F)),\,F \in D^b(X, \cB).
\end{equation*}
Through this extension, we induce a Mumford-Takemoto slope $\mu_{\omega, \beta}$ on $\Coh(X, \cB)$, yielding a category $\Coh^{\beta}(X, \cB)$ obtained by tilting $\Coh(X, \cB)$ using the slope $\mu_{\omega, \beta}$ as in Section \ref{BMSZSec}; this is well-defined by \cite{Schmidt_blowups}, Lemma 4.4.

Similarly, the forgetful Chern character induces a slope function $\nu_{\alpha, \beta}$ on $\Coh^{\beta}(X, \cB)$, and so a tilt $\cA^{\omega}_{\alpha, \beta}(X, \cB)$ of $\Coh^{\beta}(X, \cB)$ induced by $\nu_{\alpha, \beta}$; this is well-defined by \cite{Schmidt_blowups}, Lemma 4.5. 

Set $\cA^{\omega}_{\alpha, \beta}(\tilde{X}) := \Phi^{-1}\left(\cA^{\omega}_{\alpha, \beta}(X, \cB)\right)$.
\begin{prop}\label{SchmidtProp} The pair  
\begin{equation*}
\sigma_{\alpha, \beta}(\tilde{X}) := \left(\cA_{\alpha, \beta}(\tilde{X}), Z^{\tilde{\Gamma}}_{\alpha, \beta}\right)
\end{equation*}
defines a Bridgeland stability condition on $D^b(\tilde{X})$.
\end{prop} 
\begin{proof} By Theorems \ref{ToricBStabThm} and \ref{BMTThm}, $\left(\cA_{\alpha, \beta}(X), Z_{\alpha, \beta}\right)$ is a stability condition on $D^b(X)$. According to \cite{Schmidt_blowups}, Theorem 4.6 this implies that $\big(\cA_{\alpha, \beta}(\tilde{X}), Z^{\tilde{\Gamma}}_{\alpha, \beta}\big)$ is a stability condition on $D^b(\tilde{X})$.
\end{proof}
We have an analogue of Proposition \ref{DivisorsProp} in the case of blowups.
\begin{prop}\label{BlpDivisorsProp} Let $(X, \omega)$ satisfy the assumptions of Theorem \ref{ToricBStabThm}. Let $L$ denote a line bundle satisfying $\mu_{\omega,\beta}(L) < 0$ and 
\begin{equation*}
\cos(\hat{\theta}_{\beta}) < 0 \iff \omega \cdot (c_1(L) - \beta \omega)^2 - \frac{\omega^3}{3} > 0. 
\end{equation*}    
Define 
\begin{equation*}
\tilde{L} = \pi^*L(\delta E),\, \delta \in \Q_{ > 0} 
\end{equation*}
as a $\Q$-line bundle on $\tilde{X} = \Bl_{p} X$.

Then, for a uniform choice of the parameters $k \gg k_1 \gg 1$ appearing in the proof of Proposition \ref{SVProp}, with $k$ sufficiently divisible, for all irreducible toric \emph{divisors} $\tilde{V} \subset \tilde{X}$, with $\tilde{V} \neq E$, there is an inclusion 
\begin{equation*}
\cS_{\tilde{V}}[1] \subset \tilde{L}^k[2]  
\end{equation*}
in the heart $\cA^{k\tilde{\omega}}_{\alpha, k\beta}(\tilde{X})$ of the Bridgeland stability condition 
\begin{equation*}
\sigma^{k\omega}_{\alpha, k\beta}(\tilde{X}) = \left(\cA^{k\tilde{\omega}}_{\alpha, k\beta}(\tilde{X}), Z^{k\omega, \tilde{\Gamma}}_{\alpha, k\beta}\right)
\end{equation*}
for $0 < \alpha \leq \frac{1}{\sqrt{3}}$.
\end{prop}
\begin{proof} We work with tensor powers $\tilde{L}^k$ defining a genuine line bundle. Then 
\begin{align*}
&\Phi(\tilde{L}^k) = R\pi_* R\shHom(\cE, \tilde{L}^k)\\
&= R\pi_* R\shHom(\olo_{\tilde{X}} \oplus \olo_{\tilde{X}}(E) \oplus \olo_{\tilde{X}}(2E), \tilde{L}^k)\\
&= R\pi_* \big(\tilde{L}^k \oplus \tilde{L}^k(-E) \oplus \tilde{L}^k(-2E)\big)\\
&= R\pi_* \big(\pi^*L^k(k \delta E) \oplus \pi^*L^k((k \delta -1)E)\oplus \pi^*L^k((k \delta -2)E)\big).
\end{align*}
We choose $k > 0$ sufficiently large and divisible so that $\pi^*L^k(k \delta E)$ is a genuine line bundle and we have $k \delta -2 > 0$. Then
\begin{equation*}
R\pi_*  \pi^*L^k((k \delta - i)E) = L^k,\,i = 0, 1, 2,
\end{equation*}
and so
\begin{equation*}
\Phi(\tilde{L}^k) = (L^k)^{\oplus 3}.
\end{equation*}
Since $\Phi$ is a functor of triangulated categories, this also shows
\begin{equation*}
\Phi(\tilde{L}^k[1]) = (L^k[1])^{\oplus 3}.
\end{equation*}
By our assumption $\mu_{\omega, \beta}(L) < 0$ we have $L^k[1] \in \Coh^{k\beta}(X)$. Thus, 
\begin{equation*}
\tilde{L}^k[1] \in \Coh^{k\beta}(X, \cB).
\end{equation*} 
Moreover, as in the proof of Proposition \ref{DivisorsProp}, the condition 
\begin{equation*}
\omega \cdot (c_1(L) - \beta\omega)^2 - \frac{\omega^3}{3} > 0. 
\end{equation*}
implies $L^k[1] \in \cF^{k\omega}_{\alpha, k\beta}$, so we also have
\begin{equation*}
\Phi(\tilde{L}^k[1]) \in \cF^{k\omega}_{\alpha, k\beta}(X, \cB).
\end{equation*} 
Similarly, if $\tilde{V} \subset \tilde{X}$ is an irreducible toric divisor, 
we have
\begin{align*}
&\Phi(\tilde{L}^k(k_{\tilde{V}} \tilde{V})) = R\pi_* R\shHom(\cE, \tilde{L}^k(k_{\tilde{V}} \tilde{V}))\\
&= R\pi_* R\shHom(\olo_{\tilde{X}} \oplus \olo_{\tilde{X}}(E) \oplus \olo_{\tilde{X}}(2E), \tilde{L}^k(k_{\tilde{V}} \tilde{V}))\\
&= R\pi_* \big(\tilde{L}^k(k_{\tilde{V}} \tilde{V}) \oplus \tilde{L}^k(k_{\tilde{V}} \tilde{V} -E) \oplus \tilde{L}^k(k_{\tilde{V}} \tilde{V} -2E)\big)\\
&= R\pi_* \big(\pi^*L^k(k_{\tilde{V}} \tilde{V} + k \delta E) \oplus \pi^*L^k(k_{\tilde{V}} \tilde{V} + (k \delta -1)E)\\&\quad\oplus \pi^*L^k(k_{\tilde{V}} \tilde{V} + (k \delta -2)E)\big).
\end{align*}
Suppose $\tilde{V} \neq E$ is the proper transform of an irreducible toric divisor in $X$, so $\tilde{V} = \pi^{*} V - m E$ for $m \geq 0$ where $V \subset X$ is an irreducible toric divisor. Then, 
\begin{align*}
\pi^*L^k(k_{\tilde{V}} \tilde{V} + (k \delta - i) E) = \pi^*(L^k(k_{\tilde{V}} V))((k \delta - i - m k_{\tilde{V}})E),\,i = 0,1,2.
\end{align*}
Fix $k \gg k_{\tilde{V}} \gg 1$ sufficiently large and divisible so that $\pi^*(L^k(k_{\tilde{V}} V))( k \delta E)$ is a genuine line bundle and we have $k \delta - 2 - m k_{\tilde{V}} > 0$. Then 
\begin{equation*}
\Phi(\tilde{L}^k(k_{\tilde{V}} \tilde{V})) = (L^k(k_{\tilde{V}} V))^{\oplus 3}.
\end{equation*}
As above, by the argument in the proof of Proposition \ref{DivisorsProp}, this implies 
\begin{equation*}
\Phi(\tilde{L}^k(k_1 \tilde{V})[1]) \in \cF^{k\omega}_{\alpha, k\beta}(X, \cB) 
\end{equation*}
for all $k \gg k_1 \gg 1$ and all irreducible toric divisors $\tilde{V} \subset \tilde{X}$. 

The functor $\Phi$ is an equivalence, so we have an exact triangle in $D^b(X, \cB)$
\begin{equation*}
\Phi(\tilde{L}^k) \to \Phi(\tilde{L}^k(k_{\tilde{V}} \tilde{V})) \to \Phi(\cS_{\tilde{V}}) \to \Phi(\tilde{L}^k)[1]   
\end{equation*}
which, from our computations, is isomorphic to an exact triangle
\begin{equation}\label{DivisorPhi}
(L^k)^{\oplus 3} \to (L^k(k_{\tilde{V}} V))^{\oplus 3} \to \Phi(\cS_{\tilde{V}}) \to (L^k[1])^{\oplus 3}.   
\end{equation}
Thus, $j(\Phi(\cS_{\tilde{V}}))$ must be a nontrivial torsion sheaf, so we have 
\begin{equation*}
\Phi(\cS_{\tilde{V}}) \in \Coh^{k\beta}(X),
\end{equation*}
and, in particular, we obtain a short exact sequence 
\begin{equation*}
\Phi(\cS_{\tilde{V}}) \to \Phi(\tilde{L}^k)[1] \to \Phi(\tilde{L}^k(k_{\tilde{V}} \tilde{V}))[1]    
\end{equation*} 
in $\Coh^{k\beta}(X, \cB)$, with 
\begin{equation*}
\Phi(\tilde{L}^k)[1],\,\Phi(\tilde{L}^k(k_{\tilde{V}} \tilde{V}))[1] \in \cF^{k\omega}_{\alpha, k\beta}(X, \cB).
\end{equation*} 
This shows that 
\begin{equation*}
\Phi(\cS_{\tilde{V}}) \in \cF^{k\omega}_{\alpha, k\beta}(X, \cB)     
\end{equation*} 
and the claim follows as in the proof of Proposition \ref{DivisorsProp}.
\end{proof}
Similarly, we have an analogue of Proposition \ref{codim2Prop}.
\begin{prop}\label{Blpcodim2Prop} Fix the same assumptions and notation as in Proposition \ref{BlpDivisorsProp}. Then, for a uniform choice of the parameters $k, k_1, k_2$ appearing in the proof of Proposition \ref{SVProp}, with $k$ sufficiently divisible, for all toric \emph{curves} $\tilde{V} \subset X$ not contained in $E$, there is an inclusion 
\begin{equation*}
\cS_{\tilde{V}} \subset \tilde{L}^k[2] 
\end{equation*}
in the heart $\cA^{k\tilde{\omega}}_{\alpha, k\beta}(\tilde{X})$ of the Bridgeland stability condition 
\begin{equation*}
\sigma^{k\tilde{\omega}}_{\alpha, k\beta}(\tilde{X}) = \left(\cA^{k\tilde{\omega}}_{\alpha, k\beta}(\tilde{X}), Z^{k\omega, \tilde{\Gamma}}_{\alpha, k\beta}\right)
\end{equation*}
for $0 < \alpha \leq \frac{1}{\sqrt{3}}$.
\end{prop}
\begin{proof} Let $\tilde{V} \subset \tilde{X}$ be a toric $1$-dimensional submanifold not contained in $E$. We write $\tilde{V}$ as an intersection of toric divisors, $\tilde{V} = \tilde{V}_1 \cap \tilde{V}_2$, with $\tilde{V}_i \neq E$ for $i =1, 2$. Consider the corresponding complexes of line bundles
\begin{equation*}
\cS_{\tilde{V}_2} = [L^k \to L^k(k_2 \tilde{V}_2)],\,\cS_{\tilde{V}_1, \tilde{V}_2} = [L^k(k_1 \tilde{V}_1) \to L^k(k_1 \tilde{V}_1 + k_2 \tilde{V}_2)].
\end{equation*}
By the same argument as in the proof of Proposition \ref{BlpDivisorsProp}, we have
\begin{equation*}
\Phi(\cS_{\tilde{V}_2}),\,\Phi(\cS_{\tilde{V}_1, \tilde{V}_2}) \in \cF^{k\omega}_{\alpha, k\beta}(X, \cB).
\end{equation*}

By construction, the object $\cS_{\tilde{V}}$ corresponding to $\tilde{V}$ is given by   
\begin{equation*}
\cS_{\tilde{V}} \cong [\cS_{\tilde{V}_2} \to \cS_{\tilde{V}_1, \tilde{V}_2}],
\end{equation*} 
and, using the equivalence $\Phi$, we have 
\begin{equation*}
j(\Phi(\cS_{\tilde{V}}))\cong [j(\Phi(\cS_{\tilde{V}_2})) \to j(\Phi(\cS_{\tilde{V}_1, \tilde{V}_2}))].
\end{equation*} 
By the same argument as in the proof of Proposition \ref{codim2Prop}, $\cS_V$ must be a torsion sheaf of dimension $1$ and so we have 
\begin{equation*}
\cS_V \in \cT^{k\tilde{\omega}}_{\alpha, k\beta}(\tilde{X}) \subset \cA^{k\tilde{\omega}}_{\alpha, k\beta}(\tilde{X}).
\end{equation*}

The exact triangle defining $\cS_V$
\begin{equation*}
\cS_{V_2} \to \cS_{V_1, V_2} \to \cS_V \to \cS_{V_2}[1]
\end{equation*}
induces an exact triangle
\begin{equation*}
\cS_V \to \cS_{V_2}[1] \to \cS_{V_1, V_2}[1] \to \cS_V[1].
\end{equation*}
By the same argument as in the proof of Proposition \ref{BlpDivisorsProp}, we have 
\begin{equation*}
\cS_{V_2}[1],\, \cS_{V_1, V_2}[1]  \in \Phi^{-1}\left(\cF^{k\omega}_{\alpha, k\beta}(X, \cB)[1]\right) \subset \cA^{k\tilde{\omega}}_{\alpha, k\beta}(\tilde{X}),
\end{equation*} 
so the latter triangle must be induced by an exact sequence 
\begin{equation*}
\cS_V \to \cS_{V_2}[1] \to \cS_{V_1, V_2}[1] 
\end{equation*}
in $\cA^{k\tilde{\omega}}_{\alpha, k\beta}(\tilde{X})$, and there is an injection $\cS_V \subset \cS_{V_2}[1]$ in $\cA^{k\tilde{\omega}}_{\alpha, k\beta}(\tilde{X})$. Composing this with the injection $\cS_{V_2}[1] \subset L^k[2]$ in $\cA^{k\tilde{\omega}}_{\alpha, k\beta}(\tilde{X})$ given by Proposition \ref{BlpDivisorsProp} proves the claim.
\end{proof}
Suppose $X$ is a projective toric manifold (not necessarily a threefold satisfying the conditions of Theorem \ref{ToricBStabThm}). We consider the dHYM equation on $\tilde{X} = \Bl_p X$ with respect to the K\"ahler class 
\begin{equation*}
\tilde{\omega}_{\delta} = \omega - \delta \sin(\rho)[E]  
\end{equation*}
on the $\Q$-line bundle 
\begin{equation*}
\tilde{L}^{\vee}_{\delta} = (L(\cos(\rho)\delta E))^{\vee}   
\end{equation*}
for suitable $\rho \in (0, \frac{\pi}{2})$ to be determined and $\delta \in \Q_{>0}$ sufficiently small (we can work with rational or even real line bundles in the dHYM equation since this is defined for arbitrary classes $[\alpha_0]$). We will need the following simple result; similar computations first appeared in \cite{GiacchettoMScThesis}, Section 3.2 and in greater generality in the recent work \cite{Carlo_PosBlps}, Section 4.
\begin{lemma}\label{BlpdHYMLem} Let $X^n$ be a projective toric manifold, with fixed classes $\omega_0$, $\alpha_0$. Suppose that there exists a lift $\hat{\theta}$ of $e^{\ii \hat{\theta}}$ such that
\begin{equation*}
\hat{\theta} \in \left(\frac{n-2}{2}\pi, \frac{n}{2}\pi\right).
\end{equation*}

Then, for all sufficiently small $\delta, \rho >0$, there exists a solution of the constant Lagrangian phase equation \eqref{LagPhaseEqu} on the blowup $\tilde{X} = \Bl_p X$ with respect to 
\begin{equation*}
\tilde{\omega}_{\delta} = \tilde{\omega} - \delta \sin(\rho)[E],\,\tilde{L}^{\vee}_{\delta} = (L(\cos(\rho)\delta E))^{\vee}   
\end{equation*}
if, and only if, for all proper irreducible toric subvarieties $\tilde{V} \subset \tilde{X}$ not contained in $E$ we have
\begin{equation}\label{BlpdHYMPosIntro}
\int_{\tilde{V}} \Rea(\ii \tilde{\omega}_{\delta} + c_1(\tilde{L}^{\vee}_{\delta}))^{\dim {\tilde{V}}} - \cot(\varphi_{\delta})\Imm(\ii \tilde{\omega}_{\delta} +  c_1(\tilde{L}^{\vee}_{\delta}))^{\dim \tilde{V}} > 0,
\end{equation}
where $\varphi_{\delta}$ denotes the lifted angle corresponding to $\tilde{\omega}_{\delta}$, $\tilde{L}^{\vee}_{\delta}$.
\end{lemma}
\begin{proof} The topological angle $e^{\ii \hat{\theta}_{\delta}}$ on the blowup is given by
\begin{align*}
\int_{\tilde{X}} (\tilde{\omega}_{\delta} + \ii c_1(\tilde{L}^{\vee}_{\delta}))^n \in \R_{>0} e^{\ii \hat{\theta}_{\delta}},
\end{align*}
so we have $e^{\ii \hat{\theta}_{\delta}} = e^{\ii\hat{\theta}}(1 + O(\delta))$. It follows that if there is a lift $\hat{\theta}$ of $e^{\ii \hat{\theta}}$ such that $\hat{\theta} \in \left(\frac{n-2}{2}\pi, \frac{n}{2}\pi\right)$, then the same holds for the angle on the blowup $\hat{\theta}_{\delta}$ for all sufficiently small $\delta \in \Q_{>0}$. Similarly, we have
\begin{equation*}
\varphi_{\delta}:= \frac{n}{2}\pi - \hat{\theta}_{\delta} = \varphi(1 + O(\delta))\in (0, \pi).
\end{equation*}

The Nakai-Moishezon criterion on the blowup requires that the inequality \eqref{BlpdHYMPosIntro} holds for all irreducible subvarieties $\tilde{V} \subset \tilde{X}$. We wish to show that \eqref{BlpdHYMPosIntro} holds automatically for $\tilde{V} \subset E$ provided $\delta, \rho > 0$ are sufficiently small. 

When $\tilde{V} \subset E$, we have
\begin{align*}
&\int_{\tilde{V}}(\ii \omega_{\delta} + c_1(\tilde{L}^{\vee}_{\delta}))^{\dim {\tilde{V}}}\\
& = \delta^{\dim {\tilde{V}}}(\cos(\rho) + \ii \sin(\rho))^{\dim {\tilde{V}}}(-c_1(\olo_{\PP^{\dim {\tilde{V}}}}(-1))^{\dim {\tilde{V}}}\\
& = \delta^{\dim {\tilde{V}}} (\cos(\dim {\tilde{V}}\rho) + \ii \sin(\dim {\tilde{V}}\rho))  (-c_1(\olo_{\PP^{\dim\tilde{V}}}(-1))^{\dim {\tilde{V}}},
\end{align*}
yielding 
\begin{align*}
& \int_{\tilde{V}} \Rea(\ii \tilde{\omega}_{\delta} + c_1(\tilde{L}^{\vee}_{\delta}))^{\dim {\tilde{V}}} = \delta^{\dim {\tilde{V}}}\cos(\dim {\tilde{V}}\rho),\\
&\int_{\tilde{V}} \Imm(\ii \tilde{\omega}_{\delta} +  c_1(\tilde{L}^{\vee}_{\delta}))^{\dim \tilde{V}} = \delta^{\dim {\tilde{V}}}\sin(\dim {\tilde{V}}\rho),
\end{align*}
and, since we may assume $\sin(\dim {\tilde{V}}\rho) > 0$, our condition becomes
\begin{equation*}
\cot\left( \dim {\tilde{V}} \rho\right) > \cot(\varphi_{\delta}),
\end{equation*}
which, for sufficiently small $\delta$, is equivalent to
\begin{equation*}
\cot\left( \dim {\tilde{V}}\rho\right) > \cot(\varphi).
\end{equation*}
In the interval $\varphi \in (0, \pi)$, this is equivalent to
\begin{equation*}
\varphi > (n-1)\rho.
\end{equation*}
Since $\varphi > 0$, the latter condition is satisfied for all sufficiently small $\rho \in (0, \frac{\pi}{2})$, such that $\cos(\rho)$ is rational.
\end{proof}
\begin{rmk}\label{ProperTransRmk} Note that, when $\tilde{V}$ is the proper transform of an irreducible subvariety $V \subset X$, the inequality \eqref{BlpdHYMPosIntro} is equivalent to the corresponding inequality along $V$ with respect to $\omega$, $L^{\vee}$, provided $\delta \in \Q_{>0}$ is sufficiently small, and $\delta$ can be chosen uniformly at least if there are finitely many such $\tilde{V}$ to check, as in the toric case. 
\end{rmk}
In the following we restrict to the case $\alpha = \frac{1}{\sqrt{3}}$, $\beta = 0$ for simplicity (this is enough for our application to special Lagrangians).
\begin{cor}\label{BlpBstabdHYMCor} Let $(X, \omega)$ satisfy the assumptions of Theorem \ref{ToricBStabThm}. Let $L$ denote a line bundle satisfying $\mu_{\omega}(L) < 0$ and 
\begin{equation*}
\cos(\hat{\theta}) < 0 \iff \omega \cdot (c_1(L))^2 - \frac{\omega^3}{3} > 0. 
\end{equation*}    
Define
\begin{equation*}
\tilde{\omega}_{\delta} = \tilde{\omega} - \delta \sin(\rho)[E],\,\tilde{L}^{\vee}_{\delta} = (L(\cos(\rho)\delta E))^{\vee}   
\end{equation*}
on $\tilde{X} = \Bl_{p} X$.

Suppose that 
\begin{enumerate}
\item[$(i)$] $(\tilde{L}_{\delta})^k[2]$ is \emph{semistable} with respect to the Bridgeland stability condition $\sigma^{k\tilde{\omega}}_{1/{\sqrt{3}}, 0}(\tilde{X})$ for sufficiently large $k > 0$,
\item[$(ii)$] the K\"ahler class $[\omega]$ is \emph{generic} in the same sense of Corollary \ref{BStabdHYMCor}, $(ii)$.
\end{enumerate}

Then the dHYM equation is solvable on $\tilde{X}$ with respect to $\tilde{\omega}_{\delta}$, $\tilde{L}^{\vee}_{\delta}$ for all sufficiently small $\delta, \rho > 0$.
\end{cor}
\begin{proof} Firstly, by the condition $\cos(\hat{\theta}) < 0$, there is a lift $\hat{\theta} \in \left(\frac{n-2}{2}\pi, \frac{n}{2}\pi\right)$. 

Suppose $\tilde{L}^k[2]$ is semistable with respect to $\sigma^{k\tilde{\omega}}_{1/\sqrt{3}, 0}(\tilde{X})$. According to Proposition \ref{BlpDivisorsProp}, we must have
\begin{equation*}
\arg\left(Z^{\tilde{\Gamma}, k\tilde{\omega}}_{1/\sqrt{3}, 0}\big(\cS_{\tilde{V}}[1]\big)\right) \leq \arg\left(Z^{\tilde{\Gamma},k\tilde{\omega}}_{1/\sqrt{3}, 0}\big(\tilde{L}^k[2]\big)\right)   
\end{equation*}
for all divisors $\tilde{V} \subset \tilde{X}$ given by the proper transform of $V \subset X$. By \cite{Schmidt_blowups}, Lemma 4.3, we have
\begin{equation*}
Z^{\tilde{\Gamma}, k\tilde{\omega}}_{1/\sqrt{3}, 0}\big( F \big) = \frac{1}{3} Z^{k\omega}_{1/\sqrt{3}, 0}(j(\Phi(F))) 
\end{equation*}  
for all $F \in D^b(\tilde{X})$. Thus, by the proof of Proposition \ref{BlpDivisorsProp}, in particular \eqref{DivisorPhi}, the above phase inequality is equivalent to
\begin{equation*}
\arg\left(Z^{k \omega}_{1/\sqrt{3}, 0}\big(\cS_{V}[1]\big)\right) \leq \arg\left(Z^{k \omega}_{1/\sqrt{3}, 0}\big(L^k[2]\big)\right),
\end{equation*} 
and so, according to Proposition \ref{SVProp}, we have 
\begin{equation*} 
\int_V \Rea(\ii \omega + c_1(L^{\vee}))^{\dim V} - \cot(\varphi)\Imm(\ii \omega +  c_1(L^{\vee}))^{\dim V} \geq 0.
\end{equation*}
If $[\omega]$ is generic, we obtain the strict inequality
\begin{equation*} 
\int_V \Rea(\ii \omega + c_1(L^{\vee}))^{\dim V} - \cot(\varphi)\Imm(\ii \omega +  c_1(L^{\vee}))^{\dim V} > 0,
\end{equation*}
and so, by Remark \ref{ProperTransRmk}, we have 
\begin{equation*} 
\int_{\tilde{V}} \Rea(\ii \tilde{\omega}_{\delta} + c_1(\tilde{L}^{\vee}_{\delta}))^{\dim {\tilde{V}}} - \cot(\varphi_{\delta})\Imm(\ii \tilde{\omega}_{\delta} +  c_1(\tilde{L}^{\vee}_{\delta}))^{\dim \tilde{V}} > 0,
\end{equation*}
for $\delta, \rho > 0$ sufficiently small, uniformly in $V$.  

By Proposition \ref{BlpDivisorsProp} and its proof, the same argument applies when $\tilde{V} \subset \tilde{X}$ is a toric curve not contained in $E$. 

Thus, Lemma \ref{BlpdHYMLem} shows that constant Lagrangian phase equation \eqref{LagPhaseEqu} is solvable on $\tilde{X} = \Bl_p X$ with respect to $\tilde{\omega}_{\delta}$, $\tilde{L}^{\vee}_{\delta}$.
\end{proof}
Through the equivalence $D^b(\tilde{X}) \cong \FS(\tilde{\cY}_{q_k})$, we transfer the stability condition $\sigma^{k\tilde{\omega}}_{1/{\sqrt{3}}, 0}(\tilde{X})$ to $\FS(\tilde{\cY}_{q_k})$. We denote this stability condition on $\FS(\tilde{\cY}_{q_k})$ by $\sigma^{\vee} = (\cA^{\vee}, Z^{\vee})$. For fixed $k \gg 1$, we write $\cL$, $\cL_{\tilde{V}}$ respectively for a SYZ transform of $((\tilde{L}_{\delta})^k, h)$ and the corresponding potential destabilisers $\cS_{\tilde{V}}$, in the sense of Proposition \ref{SVProp}, when $\tilde{V} \subset \tilde{X}$ is not contained in the exceptional divisor $E$. 
\begin{thm}\label{BlpThm} Let $(X, \omega)$ satisfy the assumptions of Theorem \ref{ToricBStabThm}. Let $L$ denote a line bundle satisfying $\mu_{\omega}(L) < 0$ and 
\begin{equation*}
\cos(\hat{\theta}) < 0 \iff \omega \cdot (c_1(L))^2 - \frac{\omega^3}{3} > 0. 
\end{equation*}    
Let
\begin{equation*}
\tilde{\omega}_{\delta} = \tilde{\omega} - \delta \sin(\rho)[E],\,\tilde{L}^{\vee}_{\delta} = (L(\cos(\rho)\delta E))^{\vee}   
\end{equation*}
on $\tilde{X} = \Bl_{p} X$ for sufficiently small $\delta, \rho > 0$ with $\cos(\rho) \in \Q$. Define the graded Lagrangian section $\hat{\cL}$ in the mirror $\tilde{\cU}_{q_k}$ of $(\tilde{X}, k\tilde{\omega}_{\delta})$,
\begin{equation*}
\hat{\cL} = \cL[2] = \cL((\tilde{L}_{\delta})^{ k}, h)[2], 
\end{equation*}
as a shift of the SYZ transform $\cL((\tilde{L}_{\delta})^{ k}, h) \subset \tilde{\cU}_{q_k} \subset \tilde{\cY}_{q_k}$ for $k > 0$ sufficiently large and divisible (depending only on $(X, \omega)$ and $L$), with respect to a torus-invariant representative of $k[\tilde{\omega}_{\delta}]$. Write $\tilde{V} \subset \tilde{X}$ for irreducible toric subvarieties not contained in the exceptional divisor $E$. Then 
\begin{enumerate}
\item[$(i)$] there are inclusions $\cL_{\tilde V}[\dim {\tilde V} - 1] \subset \hat{\cL}$ in the heart $\cA^{\vee}$ of the Bridgeland stability condition $\sigma^{\vee}$ on $\FS(\tilde{\cY}_{q_k})$.
\item[$(ii)$] If $[\tilde{\omega}_{\delta}] \in H^{1,1}(X, \R)$ is \emph{generic} in the sense of Corollary \ref{BStabdHYMCor}, $(ii)$ and $\hat{\cL}$ is \emph{semistable} with respect to the stability condition $\sigma^{\vee}$ on $\FS(\cY_{q_k})$, then the inequalities 
\begin{equation*}
\arg Z^{\vee}\big(\cL_{\tilde V}[\dim V - 1]\big) \leq \arg Z^{\vee}\big(\hat{\cL}\big)
\end{equation*}
hold, and $\hat{\cL}$ is isomorphic in $\FS(\cY_{q_k})$ to a \emph{special} Lagrangian SYZ section.
\item[$(iii)$] If $\tilde{X}$ is weak Fano, the central charges of $\hat{\cL}$ and $\cL_{\tilde{V}}$ are given by periods of the holomorphic volume form, up to a small correction term 
\begin{align*}
&Z^{\vee}(\hat{\cL}) = \frac{1}{3}\frac{1}{(2\pi \ii)^{n}} \int_{[\hat{\cL}]} e^{-W(k \omega)/z} \Omega_0\,(1 + O(k^{-1}, \delta)),\\
&Z^{\vee}(\cL_{\tilde V}) = \frac{1}{3}\frac{1}{(2\pi \ii)^{n}} \int_{[\cL_{\tilde V}]} e^{-W(k \omega)/z} \Omega_0\,(1 + O(k^{-1},\delta)).
\end{align*}
\item[$(iv)$] When $\tilde{X}$ is not necessarily weak Fano, there exist a Landau-Ginzburg potential $\tilde{W}_{q_k}$, \emph{complex} cycles 
\begin{equation*}
\Gamma_{\hat{\cL}},\,\Gamma_{\cL_{\tilde{V}}} \in H_n(\tilde{\cY}_{q_k}, \{\Rea(\tilde{W}_{q_k}) \gg 0\}; \Z) \otimes \C, 
\end{equation*}
and a holomorphic volume form $\tilde{\Omega}^{(k)}$, such that  
\begin{align*} 
& Z^{\vee}(\hat{\cL}) = \frac{1}{3}\frac{1}{(2\pi \ii)^{n}}\int_{\Gamma_{\hat{\cL}}} e^{-\tilde{W}(k\tilde{\omega}_{\delta})/z}\tilde{\Omega}^{(k)} (1 + O(k^{-1},\delta)),\\  
& Z^{\vee}(\cL_{\tilde V}) = \frac{1}{3}\frac{1}{(2\pi \ii)^{n}}\int_{\Gamma_{\cL_{\tilde V}}} e^{-\tilde{W}(k\tilde{\omega}_{\delta})/z}\tilde{\Omega}^{(k)} (1 + O(k^{-1},\delta)),
\end{align*}
in the sense of asymptotic expansions for $z \to 0^+$. 
\end{enumerate}
\end{thm}
\begin{proof} Claim $(i)$ follows from the definition of $\sigma^{\vee}$ through Propositions \ref{BlpDivisorsProp} and \ref{Blpcodim2Prop}. Claim $(ii)$ follows from Corollary \ref{BlpBstabdHYMCor}. As in the proof of Corollary \ref{BlpBstabdHYMCor}, we have
\begin{align*}
& Z^{\tilde{\Gamma},k\tilde{\omega}}_{1/\sqrt{3}, 0}(\tilde{X})\big(\hat{L}^k[2]\big)  = \frac{1}{3} Z^{\omega}_{1/\sqrt{3}, 0}\big(L^k[2]\big) = \frac{1}{3}\int_X e^{-\ii k \omega} \ch(L^k),\\
& Z^{\tilde{\Gamma}, k\tilde{\omega}}_{1/\sqrt{3}, 0}(\tilde{X})\big(\cS_{\tilde{V}}\big) = \frac{1}{3}Z^{\omega}_{1/\sqrt{3}, 0}\big(\cS_{V}\big) = \frac{1}{3}\int_X e^{-\ii k \omega} \ch(\cS_{V}).    
\end{align*}
In turn 
\begin{align*}
&\int_X e^{-\ii k \omega} \ch(L^k) = \int_{\tilde{X}} e^{-\ii k\tilde{\omega}_{\delta}} \ch(\tilde{L}^k)\big(1 + O(\delta)\big),\\
&\int_X e^{-\ii k \omega} \ch(\cS_{V}) = \int_{\tilde{X}} e^{-\ii k\tilde{\omega}_{\delta}} \ch( \cS_{\tilde{V}})\big(1 + O(\delta)\big).
\end{align*}
Claims $(iii)$ and $(iv)$ now follow respectively from \cite{J_toricThomasYau}, Theorems 2.5 and 2.20.
\end{proof}
\section{Fibrations}\label{RelSec}
In this Section we prove Theorem \ref{AdiabaticThmIntro}, stated more precisely as Theorem \ref{AdiabaticThm}.
\begin{prop}\label{AdiabaticdHYMProp} Let $X^n$ be a toric manifold with a toric submersion 
\begin{equation*}
f\!: X \to \PP^1.
\end{equation*} 
Consider K\"ahler classes on $X$ of the form
\begin{equation*}
\omega_m = m f^{*}\omega_{\PP^1} + \omega_X, 
\end{equation*}
where $\omega_X$ is relatively K\"ahler and $m > 0$ is sufficiently large. Fix a line bundle $L$ on $X$ such that for the dual $L^{\vee}$ we have $\varphi_X \in (\frac{\pi}{2}, \pi) \mod 2\pi$. Then, for all $m \gg 1$, the constant Lagrangian phase equation \eqref{LagPhaseEqu} on $L^{\vee}$ is solvable on $X$ iff it is solvable on the fibres of $f$ with respect to the restrictions of $\omega_m$, $L^{\vee}$.  
\end{prop}
\begin{proof}
For a fixed line bundle $L$ on $X$, the topological angle $e^{\ii \hat{\theta}_X}$ is determined by
\begin{equation*}
\int_X (\omega_m + \ii c_1(L^{\vee}))^n \in e^{\ii \hat{\theta}_X}\R_{>0},
\end{equation*}
so we have 
\begin{align*}
&\hat{\theta}_X = \arg\left(\int_{X} f^*\omega_{\PP^1} \cdot (\omega_X + \ii c_1(L^{\vee}))^{n-1} +  m^{-1} \frac{1}{n}\int_X (\omega_X + \ii c_1(L^{\vee}))^{n} \right)\\& \mod 2\pi,
\end{align*}
\begin{align*}
&= \arg\left(\int_{F} (\omega_X + \ii c_1(L^{\vee}))^{n-1}\right) + \arg\left(1 +  m^{-1} \frac{1}{n \kappa} \frac{\int_X (\omega_X + \ii c_1(L^{\vee}))^{n}}{\int_{F} (\omega_X + \ii c_1(L^{\vee}))^{n-1}}\right)\\&\mod 2\pi
\end{align*}
for $m \gg 1$, since $[f^*\omega_{\PP^1}]$ is proportional to the class of a toric fibre $F$ up to some constant $\kappa > 0$.

Similarly, the topological angle $e^{\ii \hat{\theta}_F}$ associated with the restriction of our data on fibres $F$ is determined by 
\begin{equation*}
\int_F (\omega_m + \ii c_1(L^{\vee}))^{n-1} \in e^{\ii \hat{\theta}_F}\R_{>0},
\end{equation*}
so we have
\begin{equation*}
\hat{\theta}_F = \arg \int_F (\omega_X + \ii c_1(L^{\vee}))^{n-1} \mod 2\pi,
\end{equation*}
since $[(f^*\omega_{\PP^1})|_F] = 0$. Thus, in our situation, we have
\begin{align*}
\hat{\theta}_X = \hat{\theta}_F + \arg\left(1 +  m^{-1} \frac{1}{n \kappa} \frac{\int_X (\omega_X + \ii c_1(L^{\vee}))^{n}}{\int_{F} (\omega_X + \ii c_1(L^{\vee}))^{n-1}}\right) =  \hat{\theta}_F + O(m^{-1}).
\end{align*}
Let us consider the toric Nakai-Moishezon criterion on $X$, that is, positivity for integrals of the form 
\begin{equation*} 
\int_V \Rea(\ii \omega_m + c_1(L^{\vee}))^{\dim V} - \cot(\varphi_X)\Imm(\ii \omega_m +  c_1(L^{\vee}))^{\dim V} 
\end{equation*}
where $\varphi_{X} = \frac{n}{2}\pi - \hat{\theta}_{X} \in (0, \pi)$. 

When $V$ is a fibre of $f$, we obtain the condition
\begin{equation*} 
\int_F \Rea(\ii \omega_X + c_1(L^{\vee}))^{n-1} - \cot(\varphi_X)\Imm(\ii \omega_X +  c_1(L^{\vee}))^{n-1} > 0, 
\end{equation*}
or, equivalently, 
\begin{equation*} 
\cot(\varphi_F)  > \cot(\varphi_X). 
\end{equation*}
Using the relation
\begin{equation*}
\varphi_F = \frac{n-1}{2}\pi - \hat{\theta}_F = \varphi_X - \frac{\pi}{2} + O(m^{-1}),
\end{equation*}
implying
\begin{equation*}
\cot( \varphi_X ) = -\tan(\varphi_F) + O(m^{-1}),
\end{equation*}
we see that, for $m \gg 1$, positivity with respect to a fibre is equivalent to the condition
\begin{equation*} 
\cot(\varphi_F)  > -\tan(\varphi_F). 
\end{equation*}
In the interval $\varphi_F \in (0, \pi)$, this holds precisely when $\varphi_F \in \left(0, \frac{\pi}{2}\right)$.

On the other hand, when $V$ is not a fibre of $f$, we use the expansion in $m$ of the integral with leading order term 
\begin{align*} 
&m \int_V \left(\Rea\left((\ii f^*\omega_{\PP^1})\wedge(\ii\omega_X + c_1(L^{\vee}))^{\dim V-1}\right)\right.\\
&\left. - \cot(\varphi_X)\Imm\left((\ii f^*\omega_{\PP^1})\wedge(\ii\omega_X + c_1(L^{\vee}))^{\dim V-1}\right)\right)\\
&=  \kappa m  \int_{V \cap F} \left(\Rea\left(\ii (\ii\omega_X + c_1(L^{\vee}))^{\dim V-1}\right)\right.\\
& \left.- \cot(\varphi_X)\Imm\left(\ii (\ii\omega_X + c_1(L^{\vee}))^{\dim V-1}\right)\right). 
\end{align*}
In turn 
\begin{align*} 
&\int_{V \cap F} \Rea\left(\ii (\ii\omega_X + c_1(L^{\vee}))^{\dim V-1}\right) - \cot(\varphi_X)\Imm\left(\ii (\ii\omega_X + c_1(L^{\vee}))^{\dim V-1}\right)\\
&= \int_{V \cap F} - \Imm (\ii\omega_X + c_1(L^{\vee}))^{\dim V-1} - \cot(\varphi_X)\Rea (\ii\omega_X + c_1(L^{\vee}))^{\dim V-1}\\
&= \int_{V \cap F} - \Imm (\ii\omega_X + c_1(L^{\vee}))^{\dim V-1} + \tan(\varphi_F)\Rea (\ii\omega_X + c_1(L^{\vee}))^{\dim V-1}\\
& + O(m^{-1})\\   
&= \tan(\varphi_F) \int_{V \cap F} \left(\Rea (\ii\omega_X + c_1(L^{\vee}))^{\dim V-1}\right.\\
& \left. - \cot(\varphi_F)\Imm (\ii\omega_X + c_1(L^{\vee}))^{\dim V-1}\right) + O(m^{-1}).  
\end{align*}
When $\varphi_X \in (\frac{\pi}{2}, \pi)$, or equivalently, for sufficiently large $m$, when $\varphi_F \in \left(0, \frac{\pi}{2}\right)$, we have $\tan(\varphi_F) > 0$, and the above computation shows that the Nakai-Moishezon criterion on a toric fibre $F$ implies that on $X$ with respect to subvarieties not identical to a fibre. But we know already that for $\varphi_F \in \left(0, \frac{\pi}{2}\right)$ the criterion automatically holds with respect to fibres for $m \gg 1$.

Thus, with our assumptions, the Nakai-Moishezon criterion for the constant Lagrangian phase equation holds on $X$ iff it holds for the restriction to a toric fibre, which implies our claim.
\end{proof}

We specialise to the case when $f\!: X \to \PP^1$ is a projective toric submersion with relative dimension $d = 2, 3$, endowed with the K\"ahler class 
\begin{equation*}
\omega_m = m f^{*}\omega_{\PP^1} + \omega_X, 
\end{equation*}
where $\omega_X$ is relatively K\"ahler and $m > 0$ is sufficiently large. All fibres $X_t$ are isomorphic to a fixed toric manifold of dimension $d$. When $d = 3$ we assume that the polarised toric manifolds $(X_t, \omega_m|_{X_t})$ satisfy the assumptions of Theorems \ref{MainThmIntro} or \ref{BlpMainThmIntro}.

Then, as an application of the key existence result \cite{Bayer_families}, Theorem 1.3, there is a (relative) stability condition $\relsigma$ on $D^b(X)$ over $\PP^1$, i.e. a family of stability conditions $\sigma_t = (\cA_t, Z_t)$ on $D^b(X_t)$ over $\PP^1$ satisfying the properties of \cite{Bayer_families}, Definition 1.1. 

The crucial consequences of (and motivations for) these properties are
\begin{enumerate}
\item[$(a)$] by \cite{Bayer_families}, Theorem 1.2, the map sending $\relsigma$ to the family of central charges $Z_t$ is a local isomorphism;
\item[$(b)$] according to \cite{Bayer_families}, Theorem 1.4, there exists an algebraic stack $\cM_{\relsigma}(v)$ of finite type over $\PP^1$ (the stack of \emph{$\relsigma$-semistable objects}), with a good coarse moduli space, representing the functor of families of $\sigma_t$-semistable objects $E_t \in D^b(X_t)$ with fixed topological data $v$ in the sense of \cite{Bayer_families}, Definition 21.11 (2).
\end{enumerate}
  
\begin{thm}\label{AdiabaticThm} Let $f\!: X \to \PP^1$ be a projective toric submersion with relative dimension $d = 2, 3$, endowed with the K\"ahler class 
\begin{equation*}
\omega_m = m f^{*}\omega_{\PP^1} + \omega_X, 
\end{equation*}
where $\omega_X$ is relatively K\"ahler and $m > 0$ is sufficiently large. When $d = 3$ we assume that the fibres $(X_t, \omega_m|_{X_t})$ satisfy the assumptions of Theorems \ref{MainThmIntro} or \ref{BlpMainThmIntro}. Fix a line bundle $L$ on $X$ such that for the dual $L^{\vee}$ we have 
\begin{equation*}
\varphi = \frac{d+1}{2}\pi - \hat{\theta} \in \left(\frac{\pi}{2}, \pi\right) \mod 2\pi. 
\end{equation*}
Define the graded Lagrangian section
\begin{equation*}
\tilde{\cL} = \cL(L^k, h)[d-1] 
\end{equation*}
as a shift of the SYZ transform $\cL(L^k, h) \subset \cU_{q_k} \subset \cY_{q_k}$ for $k > 0$ sufficiently large (depending only on $(X, \omega)$ and $L$). Then 
\begin{enumerate} 
\item[$(i)$] If $L^k[d-1]$ is \emph{$\relsigma$-semistable} with respect to the stability condition $\relsigma$ on $D^b(X)$ over $\PP^1$ and $[\omega_X|_{X_t}] \in H^{1,1}(X, \R)$ is \emph{generic} in the sense of Corollary \ref{BStabdHYMCor}, $(ii)$ for some fibre $X_t$, then $\tilde{\cL}$ is isomorphic in $\FS(\cY_{q_k})$ to a \emph{special} Lagrangian SYZ section.
\item[$(ii)$] If $X_t$ is weak Fano, the central charges of $\tilde{\cL}_t := \cL(L^k_t, h_t)[d-1]$ and $\cL_{V_t}$ are given by periods of the holomorphic volume form, up to a small correction term 
\begin{align*}
&Z^{\vee}(\tilde{\cL}_t) := Z(L^k_t[d-1]) = \frac{1}{(2\pi \ii)^{d}} \int_{[\tilde{\cL}_t]} e^{-W_{X_t}(k \omega_X)/z} \Omega_0\,(1 + O(k^{-1})),\\
&Z^{\vee}(\cL_{V_t}) := Z(\cS_{V_t}) = \frac{1}{(2\pi \ii)^{d}} \int_{[\cL_{V_t}]} e^{-W_{X_t}(k \omega_X)/z} \Omega_0\,(1 + O(k^{-1})).
\end{align*}
\item[$(iii)$] When $X_t$ is not necessarily weak Fano, there exist a Landau-Ginzburg potential $W_{X_t}(k\omega_X)$, \emph{complex} cycles 
\begin{equation*}
\Gamma_{\tilde{\cL}_t},\,\Gamma_{\cL_{V_t}} \in H_n(\cY_{X_t}(k\omega_X), \{\Rea(W_{X_t}(k\omega_X)) \gg 0\}; \Z) \otimes \C
\end{equation*}
and a holomorphic volume form $\Omega^{(k)}$, such that  
\begin{align*} 
& Z^{\vee}(\tilde{\cL}_t) = \frac{1}{(2\pi \ii)^{d}}\int_{\Gamma_{\tilde{\cL}_t}} e^{-W_{X_t}(k\omega_X)/z}\Omega^{(k)} (1 + O(k^{-1})),\\  
& Z^{\vee}(\cL_{V_t}) = \frac{1}{(2\pi \ii)^{d}}\int_{\Gamma_{\cL_{V_t}}} e^{-W_{X_t}(k\omega_X)/z}\Omega^{(k)} (1 + O(k^{-1})),
\end{align*}
in the sense of asymptotic expansions for $z \to 0^+$. 
\end{enumerate}
\end{thm}
\begin{proof}
Let $d = 3$. Suppose $L^k[d-1]$ is $\relsigma$-semistable. By construction, the restriction $(L^k[d-1])_t$ to each fibre of $f$ is $\sigma_t$-semistable with respect to $\sigma_t$, where $\sigma_t$ is the stability condition on $D^b(X_t)$ appearing in Theorem \ref{BMTThm} or Proposition \ref{SchmidtProp}. 

As in the proof of Proposition \ref{AdiabaticdHYMProp}, we can fix choices of $\varphi_X$, $\varphi_{X_t}$ on the total space and fibres respectively, with respect to $\omega_m$, $L^k$ and their restrictions, such that 
\begin{equation*}
\varphi_X \in \left(\frac{\pi}{2}, \pi\right),\, \varphi_{X_t}  = \varphi_X - \frac{\pi}{2} + O(m^{-1}).
\end{equation*}  
In particular, we have $\varphi_{X_t}  \in \left(0,\frac{\pi}{2} \right)$ for $m \gg 1$.

Corollary \ref{BStabdHYMCor} shows that the $\sigma_t$-semistability of $L^k[2]$ on the fibres implies the inequalities 
\begin{equation*}
\int_{V_t} \Rea ((\ii\omega_m + c_1(L^{\vee}))|_{X_t})^{\dim V_t } - \cot(\varphi_{X_t})\Imm ((\ii\omega_m + c_1(L^{\vee}))_{X_t})^{\dim V_t } \geq 0
\end{equation*}
for all proper irreducible subvarieties $V_t \subset X_t$. This condition is independent of $t$. Moreover, if $[\omega_X]$ is generic in the sense that its restriction to any fixed fibre $X_t$ does not lie in the union of analytic subvarieties of $H^{1,1}(X_t, \R)$ cut out by
\begin{equation*}
\int_{V_t} \Rea ((\ii\omega_m + c_1(L^{\vee}))|_{X_t})^{\dim V_t } - \cot(\varphi_{X_t})\Imm ((\ii\omega_m + c_1(L^{\vee}))_{X_t})^{\dim V_t } = 0,
\end{equation*}
as $V_t \subset X_t$ ranges over irreducible toric subvarieties, then the $\sigma_t$-semistability of $L^k[2]$ must imply the strict inequalities
\begin{equation*}
\int_{V_t} \Rea ((\ii\omega_m + c_1(L^{\vee}))|_{X_t})^{\dim V_t } - \cot(\varphi_{X_t})\Imm ((\ii\omega_m + c_1(L^{\vee}))_{X_t})^{\dim V_t } > 0.
\end{equation*}
Thus, the Nakai-Moishezon criterion holds for the restriction to any fibre, and by Proposition \ref{AdiabaticdHYMProp}, under our assumptions, this implies that the supercritical Lagrangian phase equation is solvable on $X$.

The same proof works for $d = 2$, using the standard construction of stability conditions on surfaces (see e.g. \cite{BayerMacriToda_Bogomolov}, Proposition 3.1.3), and the analogue of Corollary \ref{BStabdHYMCor} in this case, which is proved in \cite{J_toricThomasYau} (Theorem 1.9). This proves our claim $(i)$.

The central charges are given by
\begin{align*} 
& Z^{\vee}(\tilde{\cL}_t) := Z(L^k_t([d-1])) = (-1)^{d-1} \int_{X_t} e^{-\ii k \omega_X} \ch(L^k_t),\\  
& Z^{\vee}(\cL_{V_t}) := Z(\cS_{V_t}) = \int_{X_t} e^{-\ii k \omega_X} \ch(\cS_{V_t}),
\end{align*}
so claims $(iii)$ and $(iv)$ follow respectively from \cite{J_toricThomasYau}, Theorems 2.5 and 2.20.
\end{proof}
\section{The unstable case on $\Bl_p \PP^3$}\label{UnstSec}
Mete \cite{Mete_P3} studies triples $(X = \Bl_p \PP^3, [\omega], [\alpha])$, satisfying the supercritical condition $\varphi \in (0, \pi) \mod 2\pi$, which are \emph{dHYM-unstable}, i.e. for which dHYM-semistability \eqref{dHYMSemiPos} does not hold, (this is a higher dimensional analogue of \cite{DatarSong_slopes}, Theorem 1.13, which also contains similar results for general projective surfaces).  

Let us write $H$, $E$ for the (pullback) hyperplane and exceptional divisor classes. Let
\begin{align*}
[\omega] = \beta H - E,\,\beta > 1,\, [\alpha] = p H - r E.  
\end{align*}   
As observed in \cite{Mete_P3}, Theorem 3.1 (3), if $0 < r < \cot(\varphi)+(\cot^2(\varphi) + 1)^{\frac{1}{2}}$ then $(X = \Bl_p \PP^3, [\omega], [\alpha])$ is dHYM-unstable, so that the constant Lagrangian phase equation \eqref{LagPhaseEqu} is not solvable.
\begin{thm}[\cite{Mete_P3} Theorem 3.1 (3)]\label{MeteThm} Under some technical assumptions (spelled out in \cite{Mete_P3}, Theorem 3.1), there exist closed $(1,1)$-currents $\alpha_{\infty}$, $\alpha'_{\infty}$, with   
\begin{equation*}
\alpha_{\infty} = \alpha'_{\infty} + (\xi - r) [E],\,[\alpha'_{\infty}] = p H - \xi E
\end{equation*}
where $\xi \in \R$, $\xi > r$, $[E]$ denotes the current of integration along $E$, such that:
\begin{enumerate}
\item[$(i)$] $\alpha'_{\infty}$ has continuous local potential and is a strong solution of the dHYM equation 
\begin{equation*}
\Rea(\alpha'_{\infty} + \ii \omega)^3 = \cot(\varphi_{\min}) \Imm(\alpha'_{\infty} + \ii \omega)^3
\end{equation*}  
on $X \setminus E$, for some new angle $\varphi_{\min} \in (0, \pi)$;
\item[$(ii)$] $\alpha_{\infty}$ is a global solution on $X$ of the equation
\begin{equation}\label{WeakdHYM}
\Rea\langle(\alpha_{\infty} + \ii \omega)^3 \rangle = \cot(\varphi_{\min}) \Imm\langle(\alpha_{\infty} + \ii \omega)^3\rangle
\end{equation}  
where $\langle\,-\,\rangle$ denotes the non-pluripolar product (in the sense of Bouck\-som-Eyssidieux-Guedj-Zeriahi); 
\item[$(iii)$] the new angle $\varphi_{\min}$ satisfies the minimality condition
\begin{align*}
&\cot(\varphi_{\min})\\
& = \sup\left\{\frac{\Rea([\alpha_{s}] + \ii \omega)^3}{\Imm([\alpha_{s}] + \ii \omega)^3}\,:\, [\alpha_s] = p H - s E,\,\Imm([\alpha_{s}] + \ii \omega)^3>0,\,s>0\right\}.
\end{align*}
\end{enumerate}
\end{thm} 

Let us consider the dHYM equation on a line bundle $L^{\vee} \to X$, so $[\alpha] = c_1(L^{\vee})$. Suppose that the supremum in Theorem \ref{MeteThm} $(iii)$ is achieved by $\bar{s} \in \Q$, $\bar{s} > r$. Then, for $k > 1$ sufficiently large and divisible, $(L^{\vee})^k(k (q - \bar{s})E)$ is a genuine line bundle, which moreover, by Theorem \ref{MeteThm} $(ii)$, supports a weak solution $h_E$ of the dHYM equation, i.e. a solution of \eqref{WeakdHYM}, with the \emph{correct} topological angle $\cot(\varphi_{\min})$, satisfying the minimality condition of Theorem \ref{MeteThm} $(iii)$.   

Since $\bar{s} > r$, there is a natural morphism $L^k \to L^k(k(\bar{s}-q)E)$ in $\Coh(X)$, corresponding to a mirror morphism $\cL \to \cL_{E}$ in $\FS(\cY_q)$, where
\begin{equation*}
\cL = \cL(L^k, h),\,\cL_E = \cL(L^k(k(\bar{s}-q)E), h^{\vee}_E)
\end{equation*} 
are SYZ transforms. By duality, $h^{\vee}_E$ is a weak solution of the dHYM equation on $L^k(k(\bar{s}-r)E)$, and so $\cL_E$ satisfies the special Lagrangian condition in the weak sense, with minimal angle $\cot(\varphi_{\min})$. 

Thus, it seems natural to expect that, possibly after a shift, the morphism $\cL \to \cL_E$ is the minimal slope semistable quotient in the Harder-Narasimhan filtration of $\cL$ with respect to a suitable Bridgeland stability condition, to be determined (in particular, such a quotient should admit a weak solution of the special Lagrangian equation with the correct phase angle, which is the case for $\cL_E$). Note that our results Theorems \ref{MainThm} and \ref{BlpThm} are not adequate to study this problem. Indeed, $X = \Bl_p \PP^3$ does not satisfy the assumptions of Theorem \ref{ToricBStabThm}, while the dHYM equation is always solvable on the base $\PP^3$, so Theorem \ref{BlpThm} does not allow examples for which it is not solvable on $X$.    
\addcontentsline{toc}{section}{References}
 
\bibliographystyle{abbrv}
 \bibliography{biblio_dHYM}

\noindent SISSA, via Bonomea 265, 34136 Trieste, Italy\\
Institute for Geometry and Physics (IGAP), via Beirut 2, 34151 Trieste, Italy\\
jstoppa@sissa.it    
\end{document}